\documentclass[reqno]{amsart}

\usepackage[ansinew]{inputenc}
\usepackage{amssymb}
\usepackage{amsmath}
\usepackage{graphicx}

\makeatletter \@addtoreset{equation}{section} \makeatother

\renewcommand\thefigure{\thesection.\@arabic\c@figure}
\renewcommand\thetable{\thesection.\@arabic\c@table}

\newtheorem{theorem}{Theorem}[section]
\newtheorem{lemma}[theorem]{Lemma}
\newtheorem{proposition}[theorem]{Proposition}
\newtheorem{corollary}[theorem]{Corollary}

\newcommand{\bb}[1]{{\mathbb #1}}
\newcommand{\mc}[1]{{\mathcal #1}}

\newcommand{\<}{\langle}
\renewcommand{\>}{\rangle}

\newcommand{\R}{\mathbb{R}}

\newcommand{\1}{\,\rlap{\small 1}\kern.13em 1}

\newcommand{\sqr}[2]{{\vcenter{\hrule height.#2pt%
                      \hbox{\vrule width.#2pt height#1pt\kern#1pt%
                            \vrule width.#2pt}%
                      \hrule height.#2pt}}}

\renewcommand{\limsup}{\mathop{\overline{\hbox{\rm lim}}}}
\renewcommand{\liminf}{\mathop{\underline{\hbox{\rm lim}}}}

\title[Boundary driven exclusion processes]{Static
  large deviations of boundary driven exclusion
  processes}

\author{J. Farfan}

\address{{\rm J. Farfan} \newline
IMPA, Estrada Dona Castorina 110,
CEP 22460 Rio de Janeiro, Brasil
\newline e-mail: \rm \texttt{jonathan@impa.br}}

\begin{document}

\noindent  


\begin{abstract}
We prove that the stationary measure associated to a boundary driven exclusion process in any dimension satisfies a large deviation principle with rate function given by the quasi potential of the Freidlin and Wentzell theory.
\end{abstract}

\maketitle

%%%%%%%%%%%%%%%%%%%%%%%%%%%%%%%%%%%%%%%%%%%%%%%%%%%%%%%%%%%%%%%%%%%%%%%%%%%%%%%%%%%%%%%%%%%%%%%

\section{introduction}

%%%%%%%%%%%%%%%%%%%%%%%%%%%%%%%%%%%%%%%%%%%%%%%%%%%%%%%%%%%%%%%%%%%%%%%%%%%%%%%%%%%%%%%%%%%%%%%

In the last years statical and dynamical large deviations principles of boundary driven
interacting particles systems has attracted attention as a
first step in the understanding of nonequilibrium thermodynamics (cf.
\cite{BDGJL7, bd, Der} and references therein). One of the main dificulties is that in general the stationary measure is not known explicitly and moreover it can present long range correlations (cf. \cite{S}).

In particular, many results concerning large deviations principle for the stationary measure in the context of the one dimensional boundary driven SSEP has been established.

Derrida, Lebowitz and Speer \cite{DLS} proved that the large deviation functional of the stationary state may be  defined by a time independent variational formula. However, the use of exact computations to obtain this result raises many problems for the generalization to a broader class of models.

Inspired on the Freidlin and Wentzell's theory \cite{FW}, Bertini, et al \cite{BDGJL} proved that the large deviation functional obtained in \cite{DLS} coincides with the quasi potential of the dynamical rate function.

Finally, by following the Freidlin and Wentzell's strategy \cite{FW}, Bodineau and Giacomin \cite{BG} proved directly that the quasi potential of the dynamical rate function is the large deviation functional of the stationary state. This approach relies on the fact that there is a good dynamical large deviation principle \cite{DS} together with  some good properties of the dynamical rate function and the hydrodynamic equation. For this reason, it seems to be very promising for the generalization to a large class of interacting exclusion systems.

In this article our main goal is to prove that the quasi potential of the dynamical rate function is the large deviation functional of the stationary state in the context of boundary driven exclusion process in any dimension.

We follow \cite{FW} and more closely the approach given in \cite{BG}. In fact, the arguments presented in \cite{BG} can be adapted modulo technical dificulties to our context. However there is a case not considered in the proof of the upper bound in \cite{BG}, which we describe in detail in the following.

For a fixed closed set $\mc C$ in the weak topology not containing the stationary density $\bar\rho$, small neighborhoods $\mc V_{\delta}$ (which depends on a parameter $\delta>0$) of $\bar\rho$ are considered. By following the Freidlin and Wentzell strategy, the proof of the upper bound is reduced to prove that the minimal quasi potential of densities in $\mc C$ can be estimated from above by the minimal dynamical rate function of trajectories which start at $\mc V_{\delta}$ and touch $\mc C$ before a time $T=T_{\delta}$ (which also depends on $\delta$).

At this point, in \cite{BG}, it is supposed that the time $T = T_{\delta}$ is fixed and then, by a direct aplication of the dynamical large deviation upper bound, the desired result is obtained. The same argument still works if we assume the existence of a sequence of parameters $\delta_n\downarrow 0$ with the sequence of times $T_{\delta_n}$ bounded. The problem here, is that such sequence doesn't necessarily exist. Moreover, by the construction of such times $T_{\delta}$, it is expected that $T_{\delta}\to\infty$ as $\delta\downarrow 0$.

In our context, to solve this missing case, we first prove that long trajectories which have their dynamical rate functions uniformly bounded have to be close in some moment to the stationary density $\bar\rho$ in the $L^2$ metric, and then we prove that the quasi potential is continuous at the stationary density $\bar\rho$ in the $L^2$ topology.
The proof relies on having a good dynamical large deviation principle \cite{DS}, and on some properties of the dynamical rate function and the weak solutions of the hydrodynamic equation.

In this way, we fulfill the gap in \cite{BG} described above and extend its result for a broader class of models.

Another contribution of this work is a direct proof of the lower semicontinuity of the quasi potential. In the context of one dimensional boundary driven SSEP, the lower semicontinuity of the quasi potential was obtained indirectly by using its equivalent formulation (cf. \cite{BDGJL}) in terms of a time independent variational problem introduced in \cite{DLS}.

%%%%%%%%%%%%%%%%%%%%%%%%%%%%%%%%%%%%%%%%%%%%%%%%%%%%%%%%%%%%%%%%%%%%%%%%%%%%%%%%%%%%%%%%%%%%%%%%

\section{Notation and Results}

%%%%%%%%%%%%%%%%%%%%%%%%%%%%%%%%%%%%%%%%%%%%%%%%%%%%%%%%%%%%%%%%%%%%%%%%%%%%%%%%%%%%%%%%%%%%%%%%

%%%%%%%%%%%%%%%%%%%%%%%%%%%%%%%%%%%%%%%%%%%%%%%%
\subsection{Boundary driven exclusion processes}
%%%%%%%%%%%%%%%%%%%%%%%%%%%%%%%%%%%%%%%%%%%%%%%%

Fix an integer $d \geq 1$. For each integer $N\geq 1$, let $\Omega_N=\{-N+1,\dots,N-1\}
\times\{0,\dots,N-1\}^{d-1}$ be the microscopic space and let $X_N=\{0,1\}^{\Omega_N}$ be the
configuration space. The elements of $X_N$ are denoted by $\eta$ so that
$\eta(x)=1$, resp. $0$, if the site $x$ is occupied, resp. empty, for the configuration $\eta$.
For $x,y\in \Omega_N$, we denote by $\eta^{x,y}$, resp. by $\eta^x$, the configuration obtained
from $\eta$ by exchanging the occupations of sites $x$ and $y$, resp. by flipping the occupation
variable at site $x$:

\begin{eqnarray*}
\eta^{x,y}(z)= 
\left\{ 
\begin{array}{lll}
\eta (y) & \hbox{if}\ z= x \,, \\
\eta(x) & \hbox{if}\ z=y \,, \\
\eta(z) & \hbox{if}\ z\neq x,y \,.
\end{array}
     \right.
\qquad
\eta^x(z)= \left\{ \begin{array}{ll}
1-\eta(x) & \hbox{if}\ z = x \, ,\\
\eta (z) & \hbox{if}\ z\neq x\, .\end{array}
\right.
\end{eqnarray*}

Let $\Omega = (-1,1)\times\bb T^{d-1}$ be the macroscopic space, where $\bb T^k$ is the $k$-
dimensional torus $[0,1)^k$, and denote its boundary by $\Gamma = \{-1,1\}\times\bb T^{d-1}$.
Fix a function $b:\Gamma\to(0,1)$ of class $\mc C^2$.

The boundary driven symmetric exclussion process is the Markov process on $X_N$ with generator
$$
\mc L_N=\mc L_{N,0}+\mc L_{N,b}\, ,
$$
where $\mc L_{N,0}$ corresponds to the bulk dynamics and $\mc L_{N,b}$ to the boundary dynamics.

The action of the generator $\mc L_{N,0}$ on functions $f: X_N\to \bb R$ is given by
$$
(\mc L_{N,0}f)(\eta)=N^2\sum_{i=1}^{d}\sum_{x}r_{x,x+e_i}(\eta)
\big[f(\eta^{x,x+e_i})-f(\eta)\big],
$$
where $(e_1,\dots,e_d)$ is the canonical basis of ${\bb R}^d$ and where the second sum is
performed over all $x\in {\bb Z}^d$ such that $x,x+e_i\in \Omega_N$. Moreover, for some fixed $a>-\frac{1}{2}$, the rates $r_{x,x+e_i}(\eta)$ are given by 
$$
r_{x,x+e_i}(\eta)=1+a[\eta(x-e_i)+\eta(x+2e_i)]\, ,
$$
if $x-e_i,x+2e_i\in \Omega_N$ and by
\begin{eqnarray*}
r_{x,x+e_1}(\eta) = 1+a\left[b\left(\frac{x-e_1}{N}\right)+\eta(x+2e_1)\right] && \hbox{ if } x\in\Gamma_N^-\, ,
\\
r_{x,x+e_1}(\eta) = 1+a\left[\eta(x-e_1)+b\left(\frac{x+2e_1}{N}\right)\right] && \hbox{ if } x+e_1\in\Gamma_N^+\, ,
\end{eqnarray*}
where $\Gamma_N^-$ ,resp. $\Gamma_N^+$, stands for the left, resp. right, ``boundary'' of
$\Omega_N$:
\begin{eqnarray*}
\Gamma_N^{\pm} &=& \left\{x=(x_1,\dots,x_d)\in\Omega_N: x_1=\pm(N-1)\right\}\, .
\end{eqnarray*}

The action of the generator $\mc L_{N,b}$ on functions $f:X_N\to\bb R$ is given by
$$
(\mc L_{N,b}f)(\eta)=N^2\sum_{x\in \Gamma_N}C^b_x(\eta)\big[f(\eta^x)-f(\eta)\big]\, ,
$$
where $\Gamma_N = \Gamma_N^-\cup\Gamma_N^+$ and, for $x\in \Gamma^{\pm}_N$, the rate $C^b_x(\eta)$ is given by
$$
C^b_x(\eta)=\eta(x)\left[1-b\left(\frac{x\pm e_1}{N}\right)\right]+[1-\eta(x)]
b\left(\frac{x\pm e_1}{N}\right)\, .
$$

Notice that the generators are speeded up by $N^2$, which corresponds to the diffusive scaling.
Denote by $\{\eta_t : t\geq 0\}$ the Markov process on $X_N$ associated to the
generator $\mc L_N$ and by $\bb P_{\eta}$ its distribution if the initial configuration is $\eta$.
Note that $\bb P_{\eta}$ is a probability measure on the path space $D(\bb R_+,X_N)$, which we consider endowed with the Skorohod topology. Denote also by $\bb E_{\eta}$ the expectation with respect to $\bb P_{\eta}$.

%%%%%%%%%%%%%%%%%%%%%%%%%%%%%%%%%%%%%%%%%%%%%%%%
\subsection{Hydrostatics}
\label{sechyds}
%%%%%%%%%%%%%%%%%%%%%%%%%%%%%%%%%%%%%%%%%%%%%%%%

Let $\mc M=\mc M(\Omega)$ be the space of positive measures on $\Omega$ with total mass bounded by $2$ endowed with the weak topology. For each integer $N>0$, let $\pi^N:X_N\to\mc M$ be the function which associates to a configuration $\eta$ the positive measure obtained by assigning mass $N^{-d}$ to each particle of $\eta$,
$$
\pi^N(\eta)=\frac{1}{N^d}\sum_{x\in\Omega_N}\eta(x)\delta_{x/N}\, ,
$$
where $\delta_u$ is the Dirac measure on $\Omega$ concentrated on $u$.

Let $L^2(\Omega)$ be the Hilbert space of functions $G:\Omega
\to \bb C$ such that $\int_\Omega | G(u) |^2 du <\infty$ equipped with
the inner product
\begin{equation*}
\<G,J\>_2 =\int_\Omega G(u) \, {\bar J} (u) \, du\; ,
\end{equation*}
where, for $z\in\bb C$, $\bar z$ is the complex conjugate of $z$ and
$|z|^2 =z{\bar z}$. The norm of $L^2(\Omega)$ is denoted by $\|
\cdot \|_2$.

Let $H^1(\Omega)$ be the Sobolev space of functions $G$ with
generalized derivatives $\partial_{u_1} G, \dots , \partial_{u_d} G$
in $L^2(\Omega)$. $H^1(\Omega)$ endowed with the scalar product
$\<\cdot, \cdot\>_{1,2}$, defined by
\begin{equation*}
\<G,J\>_{1,2} = \< G, J \>_2 + \sum_{j=1}^d
\<\partial_{u_j} G \, , \, \partial_{u_j} J \>_2\;,
\end{equation*}
is a Hilbert space. The corresponding norm is denoted by
$\|\cdot\|_{1,2}$. For each $G$ in $H^1(\Omega)$ we denote by $\nabla G$ its generalized gradient:
$\nabla G=(\partial_{u_1}G,\ldots, \partial_{u_d}G)$.

\renewcommand{\labelenumi}{({\bf S\theenumi})}

Let $\overline\Omega=[-1,1]\times\bb T^{d-1}$ and denote by $\mc C^m_0 (\overline{\Omega})$, $1\leq m\leq
+\infty$,  the space of
$m$-continuously differentiable real functions defined on $\overline{\Omega}$ which
vanish at the boundary $\Gamma$. Let $\varphi:[0,1]\to \bb R_+$ be given by $\varphi (r) = r (1+ ar)$ and let $\| \cdot \|$ be the Euclidean norm:
$\|(v_1,\ldots,v_d ) \|^2 =\sum_{1\le i\le d} v_i^2$.  A function
$\rho :\Omega \to [0,1]$ is said to be a weak solution of the elliptic
boundary value problem
\begin{equation}
\label{f01}
 \left\{ \begin{array}{lll}
 \Delta \varphi(\rho) \; =\; 0 \, & \hbox{ on } \Omega\, ,\\ 
\rho \;=\; b & \hbox{ on } \Gamma\, ,
\end{array}
     \right. 
\end{equation}
if
\begin{enumerate}
\item $\rho$ belongs to $H^1(\Omega)$:
\begin{equation*}
\int_\Omega {\parallel\nabla \rho(u)\parallel}^2 
du \;<\; \infty \; .
\end{equation*}

\item For every function $G\in {\mc C}^{2}_0\left(\overline\Omega \right)$,
\begin{equation*}
\int_\Omega \big(\Delta G\big)(u) \, \varphi \big(\rho(u)\big) \, du 
=\int_\Gamma \varphi(b(u)) \, \text{\bf n}_1 (u) \,
(\partial_{u_1}  G) (u) \text{d} \text{S}\; ,
\end{equation*}
where {\bf n}=$(\text{\bf n}_1,\ldots ,\text{\bf n}_d)$ stands for the
outward unit normal vector to the boundary surface $\Gamma$ and
$\text{d} \text{S}$ for an element of surface on $\Gamma$.
\end{enumerate}
Existence and uniqueness of weak solutions of equation \eqref{f01} have been established in \cite{FLM}, Section 7. We denote by $\bar\rho$ the unique weak solution of the elliptic boundary value problem \eqref{f01}.

It is clear that, for fixed $N>0$, the Markov process $\eta_t$ is irreducible. Hence, it has a unique stationary measure $\mu^N_{ss}$ on $X_N$. Let us introduce $\mc P_N=\mu^N_{ss}\circ(\pi^N)^{-1}$, which is a probability measure on $\mc M$ and describes the behavior of the empirical measure under the invariant measure. In \cite{FLM}, it has been established a law of large numbers for the empirical measure under $\mu^N_{ss}$. This hydrostatic result is equivalent to the next convergence,
\begin{eqnarray}\label{hs}
\mc P_N \;\Rightarrow\;  \delta_{\bar\rho}\; ,
\end{eqnarray}
where $\Rightarrow$ stands for weak convergence of measures on $\mc M$ and $\delta_{\bar\rho}$, for the Dirac measure on $\mc M$ concentrated on $\bar\rho(u)du$.

%%%%%%%%%%%%%%%%%%%%%%%%%%%%%%%%%%%%%%%%%%%%%%%%
\subsection{Dynamical and Statical Large Deviations}
%%%%%%%%%%%%%%%%%%%%%%%%%%%%%%%%%%%%%%%%%%%%%%%%

Let $\mc M^0$ be the subset of $\mathcal{M}$ consisting of all absolutely continuous measures with respect to the Lebesgue measure with positive density bounded by $1$:

$$
\mathcal{M}^0=\big\{\pi\in\mathcal{M}:\,\pi(du)=\rho(u)du \; \hbox{ and } \; 0\leq \rho(u)\leq 1\; \hbox{ a.e.} \big\}\, .
$$

For any $T>0$, denote by $D([0,T],\mc M)$ the set of right continuous with left limits trajectories $\pi:[0,T]\to\mc M$ endowed with the Skorohod topology.
It is clear that $\mc M^0$ is a closed subset of $\mc M$ and that $D([0,T],\mc M^0)$ is a closed subset of $D([0,T],\mc M)$.

Let $\Omega_T = (0,T)\times\Omega$ and  $\overline{\Omega_T} = [0,T]\times\overline{\Omega}$. For $1\leq m,n\leq +\infty$, we denote by $\mc C^{m,n}(\overline{\Omega_T})$ the space of functions $G = G_t(u): \overline{\Omega_T}\to \bb R$ with $m$ continuous derivatives in time and $n$ continuous derivatives in space. We also denote by $\mc C^{m,n}_0(\overline{\Omega_T})$ (resp. $\mc C^{\infty}_c(\Omega_T)$) the set of functions in $\mc C^{m,n}(\overline{\Omega_T})$ (resp. $\mc C^{\infty,\infty}(\overline{\Omega_T})$) which vanish at $[0,T]\times\Gamma$ (resp. with compact support in $\Omega_T$).

Let the energy $\mc Q_T:D([0,T],\mc M)\to[0,+\infty]$ be the functional given by 
\begin{eqnarray*}
\mc Q_T(\pi) = \sum_{i=1}^d \sup_{G\in\mc C^{\infty}_c(\Omega_T)}\left\{2\int_0^T\<\pi_t,\partial_{u_i}G_t\>\; dt- \int_0^Tdt\int_{\Omega} du\; G(t,u)^2\right\}\, .
\end{eqnarray*}

For each function $G$, let $\hat J_G = \hat
J_{G,T}:D([0,T],\mc M^0)\to\bb R$ be the functional given by
\begin{eqnarray*}
\hat J_{G}(\pi) & = & \langle\pi_T,G_T\rangle-\langle\pi_0,G_0\rangle -
\int_0^T \langle \pi_t,\partial_tG_t\rangle \, dt \\
& & -\int_0^T \langle\varphi(\rho_t),\Delta G_t\rangle \, dt
\;+\; \int_0^T dt
\int_{\Gamma^+}\varphi(b) \, \partial_{u_1}G \, dS \\
& & -\int_0^T dt \int_{\Gamma^-} \varphi(b) \, \partial_{u_1}G \, dS \;-\;
\frac{1}{2}\int_0^T \langle\sigma(\rho_t),\Vert\nabla
G_t\Vert^2\rangle \, dt \; , 
\end{eqnarray*}
where $\sigma(r)=2r(1-r)(1+2ar)$ is the mobility of the system, $\pi_t(du) =
\rho_t(u) du$ and where, for a
measure $\vartheta$ in $\mc M$ and a continuous function $G:\Omega \to \bb R$,
$\<\vartheta, G\>$ stands for the integral of $G$ with respect to $\vartheta$:
\begin{equation*}
\<\vartheta, G\> \;=\; \int_{\Omega} G(u) \, \vartheta(du)\;.
\end{equation*}
Define $J_G = J_{G,T}:D([0,T],\mc M)\to\bb R$ by
\begin{equation*}
J_G (\pi) =
\begin{cases}
\displaystyle \hat J_G (\pi) & \hbox{ if }  \pi \in D([0,T],\mc M^0) ,\\ 
+\infty & \hbox{ otherwise .}
\end{cases}
\end{equation*}

Let $I_T:D([0,T],\mc M)\to[0,+\infty]$ be the functional given by

$$
I_T(\pi) =
\begin{cases}
\displaystyle \sup_{G\in\mc C^{1,2}_0(\Omega_T)} \!\left\{J_G(\pi)\right\} & \hbox{ if } \mc Q_T(\pi)<\infty,\\
+\infty & \hbox{ otherwise }.
\end{cases}
$$

For a measurable function $\gamma:\Omega\to[0,1]$, the dynamical large deviation rate function $I_T(\cdot|\gamma):D([0,T],\mc M)\to [0,+\infty]$ is given by

$$
I_T(\pi|\gamma) =
\begin{cases}
I_T(\pi) & \hbox{ if } \pi_0(du)=\gamma(u)du\, ,\\
+\infty & \hbox{ otherwise}\,.
\end{cases}
$$

In \cite{FLM}, it has been proven that the empirical measure satisfies a dynamical large deviation principle with rate function $I_T(\cdot|\gamma)$.

Following \cite{FW, BDGJL, BG} we define $V:\mc M\to\bb [0,+\infty]$ as the quasi potential for the rate function $I_T(\cdot|\bar\rho)$:
\begin{eqnarray*}
V(\vartheta) = \inf\left\{I_T(\pi):\; T>0, \;\pi\in D([0,T],\mc M)\;\hbox{ and }\; \pi_T = \vartheta\right\}\, ,
\end{eqnarray*}
which measures the minimal cost to produce the measure $\vartheta$ starting from $\bar\rho(u)du$.
It has been proven in \cite{FLM}, Section 4, that if $I_T(\pi|\bar\rho)$ is finite then $\pi$ belongs to $C([0,T],\mc M^0)$. Therefore we may restrict the infimum in the definition of $V(\vartheta)$ to paths in $C([0,T],\mc M^0)$ and if $V(\vartheta)$ is finite, $\vartheta$ belongs to $\mc M^0$. Reciprocally, we will prove in Subsection \ref{secsrf} that $V$ is bounded on $\mc M^0$.

We are now ready to state the main result.
\begin{theorem}
The measure $\mc P_N$ satisfies a large deviation principle on $\mc M$ with speed $N^d$ and lower semicontinuous rate function $V$. Namely, for each closed set $\mc C\subset\mc M$ and each open set $\mc O\subset\mc M$,
\begin{eqnarray*}
\limsup_{N\to\infty}\frac{1}{N^d}\log\mc P_N(\mc C) \leq -\inf_{\vartheta\in \mc C}V(\vartheta) \, ,
\end{eqnarray*}
\begin{eqnarray*}
\liminf_{N\to\infty}\frac{1}{N^d}\log\mc P_N(\mc O) \geq -\inf_{\vartheta\in \mc O}V(\vartheta) \, .
\end{eqnarray*}
\end{theorem}

%%%%%%%%%%%%%%%%%%%%%%%%%%%%%%%%%%%%%%%%%%%%%%%%%%%%%%%%%%%%%%%%%%%%%%%%%%%%%%%%%%%%%%%%%%%%%%%%

\section{The Hydrodynamic Equation}

%%%%%%%%%%%%%%%%%%%%%%%%%%%%%%%%%%%%%%%%%%%%%%%%%%%%%%%%%%%%%%%%%%%%%%%%%%%%%%%%%%%%%%%%%%%%%%%%

We review here the hydrodynamic behavior and examine weak solutions of the hydrodynamic equation \eqref{f02}. We start with the former.

For a Banach space $(\bb B,\Vert\cdot\Vert_{\bb B})$ and $T>0$ we
denote by $L^2([0,T],\bb B)$ the Banach space of measurable functions
$U:[0,T]\to\bb B$ for which
\begin{equation*}
\Vert U\Vert^2_{L^2([0,T],\bb B)} \;=\; 
\int_0^T\Vert U_t\Vert_{\bb B}^2\, dt \;<\; \infty
\end{equation*}
holds.

Fix $T>0$ and a profile
$\rho_0\colon \overline{\Omega} \to [0,1]$. A measurable function
$\rho : \overline{\Omega_T} \to [0,1]$ is said to be a weak
solution of the initial boundary value problem
\begin{equation}
\label{f02}
\left\{ 
\begin{array}{l}
\partial_t \rho = \Delta \varphi\big(\rho\big)\, , \\ 
\rho (0 ,\cdot) =\; \rho_0 (\cdot) \, ,   \\
\rho (t, \cdot){\big\vert_\Gamma} =\; b (\cdot) \quad 
\text {for } 0\le t\le T \;,
\end{array}
\right. 
\end{equation}
in the layer $[0,T]\times \Omega$ if

\begin{enumerate}
\item[\bf{(H1)}] $\rho$ belongs to $L^2 \left( [0,T] , H^1(\Omega)\right)$:
\begin{equation*}
\int_0^T d s \, \left( \int_\Omega {\parallel\nabla \rho(s,u)\parallel}^2 
du \right)<\infty \; ;
\end{equation*}

\item[\bf{(H2)}] For every function $G=G_t(u)$ in
  ${\mc C}^{1,2}_0(\overline{\Omega_T})$,
\begin{align*}
& \int_\Omega du \, \big\{ G_T(u)\rho(T,u)-G_0(u)\rho_0 (u)\big\} -
\int_0^T ds \int_\Omega d u \,  (\partial_s G_s)(u)\rho(s,u) \\
& \quad =\; \int_0^T d s \int_\Omega d u \,
(\Delta G_s)(u) \varphi \big(\rho(s,u)\big) 
\; -\; \int_0^T d s \int_\Gamma \varphi(b(u)) {\text{\bf n}}_1 (u) 
(\partial_{u_1} G_s (u)) \text{d} \text{S}  \; .
\end{align*}
\end{enumerate}

Existence and uniqueness of weak solutions of equation \eqref{f02} has been established in \cite{FLM}, Section 7.

\begin{theorem}
Fix a measurable function $\rho_0:\Omega\to[0,1]$ and a sequence of configurations $\{\eta^N: N\geq 1\}$ with $\eta^N$ in $X_N$ and such that the sequence of measures $\pi^N(\eta^N)$ converges to $\rho_0$ in $\mc M$. Then, under $\bb P_{\eta^N}$ and for each $t$ in $[0,T]$, the sequence of random variables $\pi^N_t = \pi^N\circ\eta_t:D(\bb R_+,X_N)\to \mc M$ converges in probability to the deterministic measure $\rho(t,u)du$, where $\rho$ is the weak solution of the initial boundary value problem \eqref{f02}.
\end{theorem}
The proof of this result can be found in \cite{ELS2}.

%%%%%%%%%%%%%%%%%%%%%%%%%%%%%%%%%%%%%%%%%%%%%%%%%%%%%%%%%%%
\subsection{Weak solutions}
%%%%%%%%%%%%%%%%%%%%%%%%%%%%%%%%%%%%%%%%%%%%%%%%%%%%%%%%%%%

In this subsection we discuss some properties of weak solutions of \eqref{f02}. Most of them has been proved in \cite{FLM}.
The first one, which is Lemma 7.2 in \cite{FLM}, states that the $L^1(\Omega)$-norm of the difference of two weak
solutions of the boundary value problem \eqref{f02} decreases in time.
\begin{lemma}
\label{lem1-ann}
Fix two profiles $\rho_0^1$, $\rho_0^2: \Omega \to [0,1]$.  Let
$\rho^j$, $j=1$, $2$, be weak solutions of \eqref{f02} with initial
condition $\rho_0^j$. Then, $\|\rho_t^1 -\rho_t^2\|_1$ decreases in
time. In particular, there is at most one weak solution of
\eqref{f02}.
\end{lemma}

Next one establish monotonicity of weak solutions of \eqref{f02}. It is Lemma 7.4 in \cite{FLM}.
\begin{lemma}
\label{lembis-ann} 
Fix two profiles $\rho_0^1$, $\rho_0^2: \Omega \to [0,1]$.  Let
$\rho^j$, $j=1$, $2$, be the weak solutions of (\ref{f02}) with
initial condition $\rho_0^j$. Assume that there exists $s\ge 0$ such
that
\begin{equation*}
\lambda \big\{ u\in \Omega\ :\ \ 
\rho^1(s,u) \le\rho^2(s,u)  \big\}=1\; , 
\end{equation*} 
where $\lambda$ is the Lebesgue measure on $\Omega$. Then, for all
$t\ge s$
\begin{equation*}
\lambda \big\{ u\in \Omega\ :\ \ 
\rho^1(t,u) \le\rho^2(t,u)  \big\}=1\;.
\end{equation*}
\end{lemma}

\begin{corollary}\label{lem05sld}
For every $\delta>0$, there exists $\epsilon>0$ such that for all weak solution $\rho$ of \eqref{f02} with any initial profile $\rho_0$,
\begin{eqnarray*}
\epsilon\leq\rho(t,u)\leq 1-\epsilon \qquad\hbox{ for almost all }  (t,u) \hbox{ in } [\delta,+\infty)\times\overline\Omega\, .
\end{eqnarray*}

\end{corollary}
\begin{proof}
Let $\rho^0$ and $\rho^1$ be as in the statement of the previous corollary. For fixed $\delta>0$, there exists $\epsilon> 0$ such that
$$
\epsilon\leq \rho^0(t,u) \; \hbox{ and } \; \rho^1(t,u)\leq 1-\epsilon \; \hbox{ for almost all } (t,u) \hbox{ in } [\delta,\infty)\times\overline\Omega\,.
$$
This and Lemma \ref{lembis-ann} permit us to conclude the proof.
\end{proof}

Next is Lemma 7.6 in \cite{FLM}.
\begin{lemma}
\label{s02}
Fix two profiles $\rho_0^1$, $\rho_0^2: \Omega \to [0,1]$.  Let
$\rho^j$, $j=1$, $2$, be the weak solutions of (\ref{f02}) with
initial condition $\rho_0^j$. Then, 
\begin{equation*}
\int_{0}^\infty  \|\rho_t^1 -\rho_t^2 \|_1^2  \, dt
\;<\; \infty\; .
\end{equation*}
In particular,
\begin{equation*}
\lim_{t\to \infty} \|\rho_t^1 -\rho_t^2\|_1 \; =0 \; \; .
\end{equation*}
\end{lemma}

\begin{corollary}\label{lem02sld}
There is a nonnegative function $\Psi$ in $L^2(\bb R_+)$ such that for any profiles $\rho_0^1$, $\rho_0^2: \Omega \to [0,1]$, the weak solutions $\rho^j$, $j=1$, $2$ of \eqref{f02} with initial conditions $\rho_0^j$ satisfy
\begin{eqnarray*}
\Vert\rho_t^1 -\rho_t^2 \Vert_1  
\;\leq\; \Psi(t)\, ,
\end{eqnarray*}
for every $t\geq 0$.
\end{corollary}
\begin{proof}
Let $\rho^0$, resp. $\rho^1$, be the weak solution of the hydrodynamic equation \eqref{f02} with initial condition $\rho^0(0,\cdot)\equiv 0$, resp. $\rho^1(0,\cdot)\equiv 1$, and set $\Psi(t) = \Vert\rho^1_t - \rho^0_t\Vert_1$. By the previous lemma, $\Psi$ belongs to $L^2(\bb R_+)$. The last statement of the corollary follows from the monotonicity of weak solutions established in Lemma \ref{lembis-ann}.
\end{proof}

%%%%%%%%%%%%%%%%%%%%%%%%%%%%%%%%%%%%%%%%%%%%%%%%%%%%%%%
\subsection{Energy estimates}
\label{subsecenest}
%%%%%%%%%%%%%%%%%%%%%%%%%%%%%%%%%%%%%%%%%%%%%%%%%%%%%%%

We establish here an energy estimate for weak solutions in terms of the time $T$ and the $L^1$ distance between its initial profile and the stationary density $\bar\rho$.

We start by introducing some Sobolev Spaces. Let ${\mathcal C}_{c}^\infty (\Omega)$ be the set of infinitely
differentiable functions $G:\Omega \to \R$, with compact support in
$\Omega$. Recall from Subsection \ref{sechyds} the definition of the Sobolev
space $H^1(\Omega)$ and of the norm $\Vert\cdot\Vert_{1,2}$. Denote by
$H^1_0(\Omega)$ the closure of $C_c^{\infty}(\Omega)$ in
$H^1(\Omega)$. Since $\Omega$ is bounded, by Poincar\'e's inequality,
there exists a finite constant $C_1$ such that for all $G\in
H^1_0(\Omega)$
\begin{equation*}
\|G\|^2_2 \;\le\;  C_1 \| \partial_{u_1}G\|^2_2
\;\le\;  C_1 \sum_{j=1}^d \<\partial_{u_j} G \, , \, 
\partial_{u_j} G \>_2 \; .
\end{equation*}
This implies that, in $H^1_0 (\Omega)$
\begin{equation*}
\|G\|_{1,2,0} \;=\; \left\{ \sum_{j=1}^d
\<\partial_{u_j} G \, , \, \partial_{u_j} G \>_2  \right\}^{1/2}
\end{equation*}
is a norm equivalent to the norm $\|\cdot \|_{1,2}$.  Moreover, $H^1_0
(\Omega)$ is a Hilbert space with inner product given by
\begin{equation*}
\< G \, , \, J \>_{1,2,0}
\;=\; \sum_{j=1}^d
\<\partial_{u_j} G \, , \, \partial_{u_j} J \>_2 \; .
\end{equation*}

To assign boundary values along the boundary $\Gamma$ of $\Omega$ to
any function $G$ in $H^1(\Omega)$, recall, from the trace Theorem
(\cite{z}, Theorem 21.A.(e)), that there exists a continuous linear
operator $B:H^1(\Omega)\to L^2(\Gamma)$, called trace, such that $BG =
G\big|_{\Gamma}$ if $G\in H^1(\Omega)\cap \mc C(\overline{\Omega})$.
Moreover, the space $H^1_0(\Omega)$ is the space of functions $G$ in
$H^1(\Omega)$ with zero trace (\cite{z}, Appendix (48b)):
\begin{equation*}
H^1_0(\Omega) = \left\{G\in H^1(\Omega):\; BG = 0\right\}\,.
\end{equation*}

Since $\mc C^{\infty}(\overline{\Omega})$ is dense in $H^1(\Omega)$
(\cite{z}, Corollary 21.15.(a)), for functions $F,G$ in $H^1(\Omega)$,
the product $FG$ has generalized derivatives $\partial_{u_i} (FG) =
F\partial_{u_i} G+ G\partial_{u_i} F$ in $L^1(\Omega)$ and
\begin{equation}
\label{ibp}
\begin{split}
& \int_{\Omega}F(u)\, \partial_{u_1}G(u)\, du \; +\; 
\int_{\Omega}G(u) \, \partial_{u_1}F(u) \, du \\
& \quad =\;  \int_{\Gamma_+}
BF(u)\, BG(u)\, du \;-\; \int_{\Gamma_-} BF(u)\, BG(u)\, du\, .
\end{split}
\end{equation}
Moreover, if $G\in H^1(\Omega)$ and $f\in\mc C^1(\bb R)$ is such that
$f'$ is bounded then $f\circ G$ belongs to $H^1(\Omega)$ with
generalized derivatives $\partial_{u_i}(f\circ G) = (f'\circ
G)\partial_{u_i}G$ and trace $B(f\circ G) = f\circ(BG)$.

Finally, denote by $H^{-1}(\Omega)$ the dual of $H^1_0(\Omega)$.
$H^{-1}(\Omega)$ is a Banach space with norm $\Vert\cdot\Vert_{-1}$
given by
\begin{equation*}
\Vert v\Vert^2_{-1} = \sup_{G\in\mc C^{\infty}_c(\Omega)}
\left\{2\langle v,G\rangle_{-1,1} -
\int_{\Omega} \Vert \nabla G(u)\Vert^2du \right\}\, , 
\end{equation*}
where $\langle v,G\rangle_{-1,1}$ stands for the values of the linear
form $v$ at $G$.

Fix $T>0$ and a weak solution $\rho$ of \eqref{f02} with initial profile $\rho_0:\Omega\to [0,1]$. It is not too hard to prove that
\begin{eqnarray}\label{wsipb}
 \rho(0,u) = \rho_0(u) \;\;\hbox{ a.s. } \hbox{in } \Omega \;\;\;\;\hbox{ and }\;\;\; B\rho_t = b\;\; \hbox{ a.s. } \hbox{in } [0,T]\, .
\end{eqnarray}
In fact it is a straigthforward consequence of Lemma 4.1 and Corollary 4.3 in \cite{FLM}. Hence, by the integration by parts formula \eqref{ibp} and since $\rho$ is a weak solution of \eqref{f02}, for any function $G$ in $\mc C^{1,2}_0(\overline{\Omega_T})$,
\begin{eqnarray*}
\<\rho_T, G_T\> - \<\rho_0,G_0\> - \int_0^T\!\!\<\rho_t,\partial_t G_t\>\; dt = - \int_0^T\!\!\! dt\!\int_{\Omega} \!du\; \nabla\varphi(\rho_t(u))\cdot\nabla G_t(u)\, .
\end{eqnarray*}
From this and by Schwarz inequality, the functional $\partial_t \rho : C^{\infty}_c(\Omega_T) \to \bb R$ defined by
\begin{equation*}
\partial_t \rho (H) \;=\; - \int_0^T \< \rho_t, \partial_t H_t\>\, dt
\end{equation*}
satisfies
\begin{equation}\label{wsprop02}
\partial_t \rho (H) \;=\;- \int_0^T dt\int_{\Omega} du\; \nabla\varphi(\rho_t(u))\cdot\nabla H_t(u)
\end{equation}
and
\begin{equation*}
|\partial_t \rho (H)| \;\leq\; \left\{\int_0^T d s \, \left( \int_\Omega {\parallel\nabla \rho(s,u)\parallel}^2 
du \right)\right\}^{1/2}  \Vert H\Vert_{L^2([0,T], H^1_0(\Omega))}\, ,
\end{equation*}
for all $H$ in $C^{\infty}_c(\Omega_T)$. In particular, it can be extended to a bounded linear operator $\partial_t \rho :
L^2([0,T],H_0^1(\Omega)) \to \bb R$ which, by Proposition 23.7 in \cite{z} and by \eqref{wsprop02}, corresponds to the path $\{\partial_t\rho_t: 0\leq t\leq T\}$ in $L^2([0,T], H^{-1}(\Omega))$ with $\partial_t\rho_t: H^1_0(\Omega)\to \bb R$, $0\leq t\leq T$,  given by
\begin{eqnarray}\label{wsprop}
\<\partial_t\rho_t,G\>_{-1,1} = \int_{\Omega}\nabla\varphi(\rho_t(u))\cdot\nabla G_t(u) du\,.
\end{eqnarray}

For each weak solution $\rho$ of \eqref{f02}, let
\begin{eqnarray*}
\mc E_T(\rho) = \int_0^T dt\int_{\Omega} du\; \frac{\Vert\nabla\rho_t(u)\Vert^2}{\chi(\rho_t(u))} <\infty.
\end{eqnarray*}

\begin{lemma}\label{lem2}
There exists a positive constant $C$ such that for any $T>0$ and any weak solution $\rho$ of \eqref{f02} with initial profile $\rho_0:\Omega\to [0,1]$,
\begin{eqnarray*}
\mc E_T(\rho) \leq C\left\{T + \Vert\rho_0-\bar\rho\Vert_1\right\}\, .
\end{eqnarray*}

\end{lemma}
\begin{proof}
Fix $T>\delta>0$, a weak solution $\rho$ of \eqref{f02} and a function $\beta:\overline{\Omega}\to(0,1)$ of class $\mc C^2$ such that $\beta\big|_{\Gamma} = b$. Let $\epsilon>0$ such that
$$
1-\epsilon\,\leq\; \beta\;,\; \rho_t\;\leq \;\epsilon \qquad\hbox{ for every } t\geq \delta\, .
$$

Let $h:[\epsilon,1-\epsilon]^2\to \bb R$ be the smooth function given by
$$
h(x,y) = x\log\left(\frac{x}{y}\right)+(1-x)\log\left(\frac{1-x}{1-y}\right)\, .
$$
Recall that $\partial_t\rho$ belongs to
$L^2([0,T],H^{-1}(\Omega))$. We claim that
\begin{eqnarray}\label{ibp0}
\begin{aligned}
\int_{\delta}^T\langle\partial_t\rho_t,\partial_xh
(\rho_t,\beta)\rangle_{-1,1}\;dt
\;= & \;\int_{\Omega}h(\rho_{_T}(u),\beta(u))\,du
\\
& \;- \int_{\Omega}h(\rho_{\delta}(u),\beta(u))\,du\, .
\end{aligned}
\end{eqnarray}

Indeed, By \eqref{wsipb}, $\rho-\beta$ belongs to
$L^2\left([0,T],H^1_0(\Omega)\right)$ and $\partial_t(\rho-\beta) =
\partial_t\rho$ belongs to $L^2([0,T],H^{-1}(\Omega))$. Then, there
exists a sequence $\{\widetilde G^n: \,n\geq 1\}$ of smooth functions
$\widetilde G^n:\overline{\Omega_T}\to \bb R$ such that $\widetilde
G^n_t$ belongs to $\mc C^{\infty}_c(\Omega)$ for every $t$ in $[0,T]$,
$\widetilde G^n$ converges to $\rho-\beta$ in
$L^2([0,T],H^1_0(\Omega))$ and $\partial_t \widetilde G^n$ converges
to $\partial_t(\rho-\beta)$ in $L^2([0,T],H^{-1}(\Omega))$ (cf.
\cite{z}, Proposition 23.23(ii)). For each positive integer $n$, let
$G^n = \widetilde G^n+\beta$. Fix a smooth
function $\tilde h : \bb R^2\to\bb R$ with compact support
and such that its restriction to $[\epsilon,1-\epsilon]^2$ is $h$. It is
clear that
\begin{eqnarray}\label{ibpn}
\begin{aligned}
\int_{\delta}^T\langle\partial_tG^n_t, \partial_x\tilde
h(G^n_t,\beta)\rangle\;dt \;= & \;\int_{\Omega}\tilde
h(G^n_T(u),\beta(u))\,du \\
 & \;- \int_{\Omega}\tilde h(G^n_{\delta}(u),\beta(u))\,du\, .
\end{aligned}
\end{eqnarray}

On the one hand, $\partial_xh:[\epsilon,1-\epsilon]^2\to\bb R$ is given by
\begin{equation*}
\partial_xh(x,y) =
\log\left(\frac{x}{1-x}\right) -
\log\left(\frac{y}{1-y}\right)\, . 
\end{equation*}
Hence, $\partial_xh(\rho,\beta)$ and $\partial_x\tilde
h(G^n,\beta)$ belongs to
$L^2\left([\delta,T],H^1_0(\Omega)\right)$. Moreover, since $\partial_x
\tilde h$ is smooth with compact support and $G^n$ converges
to $\rho$ in $L^2([0,T],H^1(\Omega))$, $\partial_x\tilde
h(G^n,\beta)$ converges to $\partial_xh(\rho,\beta)$
in $L^2([\delta,T],H^1_0(\Omega))$. From this fact and since
$\partial_tG^n$ converges to $\partial_t\rho$ in
$L^2([0,T],H^{-1}(\Omega))$, if we let $n\to\infty$, the left hand
side in \eqref{ibpn} converges to
\begin{equation*}
\int_{\delta}^T\langle\partial_t\rho_t,
\partial_xh(\rho_t,\beta)\rangle_{-1,1}\;dt\, . 
\end{equation*}

On the other hand, by Proposition 23.23(ii) in \cite{z}, $G^n_{\delta}$,
resp. $G^n_T$, converges to $\rho_{\delta}$, resp. $\rho_T$, in $L^2(\Omega)$.
Then, if we let $n\to\infty$, the right hand side in \eqref{ibpn} goes
to
\begin{equation*}
\int_{\Omega}h(\rho_{_T}(u),\beta(u))du -
\int_{\Omega}h(\rho_{\delta}(u),\beta(u))du\, , 
\end{equation*}
which proves claim \eqref{ibp0}.

Let $F,U:[\delta,T]\times\overline\Omega\to\bb R$ be the functions given by
$F(t,u) = h(\rho(t,u),\beta(u))$ and $U(t,u) = \partial_x h(\rho(t,u),\beta(u))$.

By \eqref{ibp0} and \eqref{wsprop},
\begin{equation}\label{eq2}
\begin{aligned}
\int_{\Omega}[F(T,u)-F(\delta,u)]du \; = \;&  -\int_{\delta}^T\!\!dt\int_{\Omega}du\,\nabla\varphi(\rho_t(u))\cdot\nabla U_t(u)
\\  =\; & \int_{\delta}^T\!\!dt\int_{\Omega}du\,\frac{\varphi'(\rho_t(u))}{\chi(\beta(u))}\nabla\beta(u)\cdot\nabla\rho_t(u)
\\
& \; - \int_{\delta}^Tdt\int_{\Omega}du\,\varphi'(\rho_t(u))\frac{\Vert\nabla\rho_t(u)\Vert^2}{\chi(\rho_t(u))}\, .
\end{aligned}
\end{equation}

Let 
$$
\mc E_{[\delta,T]}(\rho)=\int_{\delta}^Tdt\int_{\Omega}du\;\frac{\Vert\nabla\rho_t(u)\Vert^2}{\chi(\rho_t(u))}\, .
$$

Since $\varphi'$ is bounded bellow on $[0,1]$ by some positive constant $C_1$, by \eqref{eq2} and the elementary inequality $2ab\leq A^{-1}a^2+Ab^2$,
\begin{eqnarray*}
2\mc E_{[\delta,T]}(\lambda) & \leq & \frac{2}{C_1}\int_{\delta}^Tdt\int_{\Omega}du\,\varphi'(\rho_t(u))\frac{\Vert\nabla\rho_t(u)\Vert^2}{\chi(\rho_t(u))}
\\ & \leq & \mc E_{[\delta,T]}(\rho) + \frac{1}{C_1^2}\int_{\delta}^Tdt\int_{\Omega}du\,\frac{\varphi'(\rho_t(u))^2\chi(\rho_t(u))}{\chi(\beta(u))^2}\Vert\nabla\beta(u)\Vert^2
\\ & & + \frac{2}{C_1}\Vert F_T-F_{\delta}\Vert_1\, .
\end{eqnarray*}
Therefore, since $\varphi'$, $\chi$ are bounded above on $[0,1]$ by some positive constant and since $\beta$ is a function in $\mc C^2(\overline\Omega)$ bounded away from $0$ and $1$, there exists a constant $C_2=C_2(\beta)$ such that
\begin{eqnarray*}
\mc E_{[\delta,T]}(\rho)\leq C_2(T-\delta)+\frac{2}{C_1}\Vert F_T-F_{\delta}\Vert_1 \, .
\end{eqnarray*}
Thus, in order to conclude the proof, we just need to show that there is a constant $C'>0$ such that
\begin{eqnarray}\label{est1sld}
\Vert F_T-F_{\delta}\Vert_1\leq C'\Vert\rho_0-\bar\rho\Vert_1\, ,
\end{eqnarray}
and then let $\delta\downarrow 0$. From the definition of $F$ and since $\beta$ is bounded away from 0 and 1 it is easy to see that $\Vert F_T-F_{\delta}\Vert_1$ is bounded above by
\begin{eqnarray}\label{exp1}
\begin{aligned}
&\int_{\Omega} \left\{|f(\rho_T(u))-f(\rho_{\delta}(u))| +|f(1-\rho_T(u))-f(1-\rho_{\delta}(u))|\right\}du
\\
&+C_3\Vert\rho_T-\rho_{\delta}\Vert_1\, ,
\end{aligned}
\end{eqnarray}
where $f(r) = r\log r$ and $C_3 = C_3(\beta)$ is a positive constant.

Fix $\delta_0>0$ such that $2\delta_0\leq\bar\rho(u)\leq 1-2\delta_0$ for all $u$ in $\Omega$. Let $A_{\delta}$ be the subset of $\Omega$ defined by
$$
A_{\delta} = \{u\in\Omega: \,|\rho_{\delta}(u)-\bar\rho(u)|>\delta_0 \;\hbox{ or }\; |\rho_T(u)-\bar\rho(u)|>\delta_0\} \, .
$$
Decompose the integral term in \eqref{exp1} as the sum of two integral terms $\int_{A_{\delta}} +\int_{A_{\delta}^c}$.

On the one hand, it is clear that $m(A_{\delta})\leq\delta_0^{-1}(\Vert\rho_{\delta}-\bar\rho\Vert_1+\Vert\rho_T-\bar\rho\Vert_1)$ and then, since $-e^{-1}\leq f(r)\leq 0$ for all $r\in(0,1]$, the first integral term is bounded above by
$$
\frac{2}{e\delta_0}\left\{\Vert\rho_{\delta}-\bar\rho\Vert_1+\Vert\rho_T-\bar\rho\Vert_1\right\}\, .
$$

On the other hand, $A_{\delta}^c\subset \{u\in\Omega : \,\delta_0\leq\rho_{\delta}(u), \rho_T(u)\leq 1-\delta_0 \}$ and there exists a constant $C_4=C_4(\delta_0)>0$ such that $|f(r)-f(s)|\leq C_4|r-s|$ for all $r,s\in[\delta_0,1]$. Hence, the second integral term is bounded by
$$
2C_4\Vert\rho_T-\rho_{\delta}\Vert_1\, .
$$

These bounds together with \eqref{exp1} and Lemma \ref{lem1-ann} prove \eqref{est1sld} and we are done.

\end{proof}

%%%%%%%%%%%%%%%%%%%%%%%%%%%%%%%%%%%%%%%%%%%%%%%%%%%%%%%%%%%%%%%%%%%%%%%%%%%%%%%%%%%%%%%%%%%%%%%%

\section{The Rate Functions}

%%%%%%%%%%%%%%%%%%%%%%%%%%%%%%%%%%%%%%%%%%%%%%%%%%%%%%%%%%%%%%%%%%%%%%%%%%%%%%%%%%%%%%%%%%%%%%%%

We discuss in this section some results concerning the dynamical and the statical rate functions. The properties for weak solutions of the  hydrodynamic equation \eqref{f02} established in the previous section play a fundamental role in the derivation of many of these results. Here and throughout the rest of this article we denote by $\overline\vartheta$ the measure in $\mc M^0$ with density given by the stationary profile $\bar\rho$, i.e., $\overline\vartheta(du) = \bar\rho(u)du$.

%%%%%%%%%%%%%%%%%%%%%%%%%%%%%%%%%%%%%%%%%%%%%%%%
\subsection{The functional $I_T$}
\label{secdrf}
%%%%%%%%%%%%%%%%%%%%%%%%%%%%%%%%%%%%%%%%%%%%%%%%
Here we study some properties of the functional $I_T$, which is closely related to the dynamical rate function $I_T(\cdot|\gamma)$. We review first the dynamical large deviation principle established in \cite{FLM}.

\begin{theorem}\label{dldp}
Fix a measurable function $\gamma:\Omega\to[0,1]$ and a sequence of configurations $\{\eta^N :\; N\geq 1\}$ with $\eta^N$ in $X_N$ such that the sequence of measures $\pi^N(\eta^N)$ converges to $\gamma(u)du$ in $\mc M$. Then the measure ${\text{\bf Q}}_{\eta^N}=\bb P_{\eta^N}\circ(\pi^N)^{-1}$ on $D([0,T],\mc M)$ satisfies a large deviation principle with speed $N^d$ and lower semicontinuous rate function $I_T(\cdot|\gamma)$. Namely, for each closed set $\mc C\subset D([0,T],\mc M)$ and for each open set $\mc{O}\subset D([0,T],\mc M)$,
\begin{eqnarray*}
\limsup_{N\to\infty}\frac{1}{N^d}\log {\text{\bf Q}}_{\eta^N}(\mc C)\leq - \inf_{\pi\in\mc C} I_T(\pi|\gamma)\, ,
\end{eqnarray*}
\begin{eqnarray*}
\liminf_{N\to\infty}\frac{1}{N^d}\log {\text{\bf Q}}_{\eta^N}(\mc{O})\geq - \inf_{\pi\in\mc{O}} I_T(\pi|\gamma)\, .
\end{eqnarray*}
Moreover, the functional $I_T(\cdot|\gamma)$ have compact level sets.
\end{theorem}

It is easy to see that the results presented in \cite{FLM}, Section 4, which holds for the dynamical rate function $I_T(\cdot|\gamma)$, also holds for the functional $I_T$. We review some of them with the functional $I_T$ in the place of $I_T(\cdot|\gamma)$.

We start with Lemma 4.1 in \cite{FLM}. It is well known that a trajectory $\pi_t(du)=\rho(t,u)du$ in $D([0,T],\mc M^0)$ has finite energy, $\mc Q_T(\pi)<\infty$, if and only if its density $\rho$ belongs to $L^2([0,T],H^1(\Omega))$, in which case,
\begin{equation*}
\mc Q_T(\pi)=
\int_0^Tdt\int_{\Omega}du\;\Vert\nabla\rho_t(u)\Vert^2 < \infty\, .
\end{equation*}
Thus, if a path $\pi_t(du)=\rho(t,u)du$ in $D([0,T],\mc M^0)$ has finite energy $\mc Q_T(\pi)<\infty$ then $\rho_t$ belongs to $H^1(\Omega)$ for almost all $0\leq t\leq T$ and so $B(\rho_t)$ is well defined for those $t$.

\begin{lemma}
\label{lem01}
Let $\pi$ be a trajectory in $D([0,T],\mc M)$ such that
$I_T(\pi)<\infty$. Then $\pi_t(du) = \rho(t,u)du$ belongs to $\mc C([0,T],\mc M^0)$ and $B(\rho_t) = b$ for almost all $t$ in $[0,T]$.
\end{lemma}

Next one is Corollary 4.3 in \cite{FLM}.

\begin{lemma}
\label{rfhe}
The density $\rho$ of a path $\pi_t(du)=\rho(t,u)du$ in $D([0,T],\mc
M^0)$ is a weak solution of the equation \eqref{f02} for some initial profile $\rho_0$ if and only if $I_T(\pi)$
vanishes.
\end{lemma}

Next result is a straightforward consequence of Corollary 4.4 and Lemma 4.5 in \cite{FLM}.

\begin{lemma}\label{lemsld}
Let $\{\pi^n(t,du) =\rho^n(t,u)du: n\geq 1\}$ be a sequence of trajectories in $D([0,T],\mc M^0)$ such that, for some positive constant $C$,
$$
\sup_{n\geq 1}\left\{I_T(\pi^n)\right\}\; \leq \; C.
$$
If $\rho^n$ converges to $\rho$ weakly in $L^2(\Omega_T)$ then
$\rho^n$ converges to $\rho$ strongly in $L^2(\Omega_T)$.
\end{lemma}

Finally, recall from Section 4 in \cite{FLM} that if $I_T(\pi)<\infty$ then $\partial_t\rho_t$ belongs to $L^2([0,T], H^{-1}(\Omega))$ and, for any function $G$ in $\mc C_0^{1,2}(\overline{\Omega_T})$,
\begin{eqnarray}\label{jflm}
\begin{aligned}
J_G(\pi)\; = \;& \int_0^T\<\partial_t\rho_t, G_t\>_{-1,1}\; dt +\int_0^T dt\int_{\Omega}du\; \nabla\varphi(\rho_t(u))\cdot\nabla G_t(u)
\\
& - \frac{1}{2}\int_0^T dt\int_{\Omega} du \;\sigma(\rho_t(u))\Vert\nabla G_t(u)\Vert^2\, .
\end{aligned}
\end{eqnarray}

Let $\mc C_1(\overline{\Omega})$ be the set of continuous functions $f:\overline{\Omega}\to \bb R$ such that
$$
\sup_{u\in\Omega}|f(u)| = 1\, .
$$
Recall that we may define a metric on $\mc M$ by introducing a dense countable family $\{f_k:\,k\geq 1\}$ of functions in $\mc C_1(\Omega)$, with $f_1=1$, and by defining the distance
\begin{eqnarray*}
d(\vartheta_1,\vartheta_2) = \sum_{k=1}^{\infty}\frac{1}{2^k}|\langle\vartheta_1,f_k\rangle - \langle\vartheta_2,f_k\rangle|\, .
\end{eqnarray*}

Let $\bb D$ be the space of measurable functions on $\Omega$ bounded below by $0$ and bounded above by $1$ endowed with the $L^2(\Omega)$ topology.
$$
\bb D = \{\rho:\Omega\to [0,1]\, : \;\, 0\leq \rho(u)\leq 1 \,\hbox{ a.e.}\}\, .
$$

For $\vartheta\in\mc M$, $\rho\in\bb D$ and $\varepsilon>0$, let us denote by $\mc B_{\varepsilon}(\vartheta)$ the open $\varepsilon$-ball in $\mc M$ with centre $\vartheta$ in the $d$-metric,
\begin{eqnarray*}
\mc B_{\varepsilon}(\vartheta)=\left\{\widetilde\vartheta\in\mc M:d\big(\widetilde\vartheta,\vartheta\big)<\varepsilon\right\},
\end{eqnarray*}
and by  $\bb B_{\varepsilon}(\rho)$ the open $\varepsilon$-ball in $\bb D$ with centre $\rho$ in the $L^2(\Omega)$ norm,
\begin{eqnarray*}
\bb B_{\varepsilon}(\rho)=\{\tilde\rho\in\bb D:\Vert\tilde\rho-\rho\Vert_2<\varepsilon\} \, .
\end{eqnarray*}

Next result states that any trajectory whose density stays a long time far away from $\bar\rho$ in the $L^2(\Omega)$ norm pays a nonnegligible cost.

For each $\delta>0$ and each $T>0$ denote by $D([0,T],\mc M^0
\backslash\bb B_{\delta}(\bar\rho))$ the set of trajectories $\pi(t,du) = \rho(t,u)du$ in $D([0,T],\mc M^0)$ such that $\rho_t\notin \bb B_{\delta}(\bar\rho)$ for all $0\leq t\leq T$.
\begin{lemma}\label{lem1}
For every $\delta>0$, there exists $T>0$ such that
$$
\inf \{I_T(\pi):\;\pi\in D([0,T],\mc M^0
\backslash\bb B_{\delta}(\bar\rho)) \} > 0\, .
$$

\end{lemma}

\begin{proof}

By Corollary \ref{lem02sld}, there exists $T_0=T_0(\delta)>0$ such that for any weak solution $\lambda$ of \eqref{f02},
\begin{eqnarray}\label{eq1}
\Vert\lambda_t-\bar\rho\Vert_2<\delta/2 \qquad\hbox{ for all } t\geq T_0\, .
\end{eqnarray}

We assert that the statement of the lemma holds with $T=2T_0$. If this is not the case, there exists a sequence of trajectories $\{\pi^k_t(du) =\rho^k(t,u)du :\,k\geq 1\}$ in $D\left([0,T],\mc M^0\backslash\bb B_{\delta}(\bar\rho)\right)$ such that $I_T(\pi^k)\leq 1/k$. Since $I_T$ has compact level sets, by passing to a subsequence if necessary, we may assume that $\pi^k$ converges to some $\pi_t(du) = \rho(t,u)du$ in $D([0,T],\mc M^0)$. Moreover, by Lemma \ref{lemsld}, $\rho^k$ converges to $\rho$ strongly in $L^2(\Omega_T)$.

On the other hand, the lower semicontinuity of $I_T$ implies that $I_T(\pi) = 0$ or equivalently, by Lemma \ref{rfhe}, that $\rho$ is a weak solution of \eqref{f02}. Hence, by \eqref{eq1} and since $\left\Vert\rho^k_t-\bar\rho\right\Vert_2\geq\delta$ for all $t\in[0,T]$ and for all positive integer $k$,
\begin{eqnarray*}
\int_{0}^T\left\Vert\rho^k_t-\rho_t\right\Vert^2_2 dt\geq \int_{T_0}^{2T_0}\left\Vert\rho^k_t-\rho_t\right\Vert^2_2 dt \geq \delta^2T_0/4\, ,
\end{eqnarray*}
which contradicts the strong convergence of $\rho^k$ to $\rho$ in $L^2(\Omega_T)$ and we are done.

\end{proof}

The same ideas permit us to establish an analogous result for the weak topology as follows.
\begin{lemma}\label{cor1}
For every $\varepsilon>0$, there exists $T>0$ such that
\begin{eqnarray*}
\inf\left\{I_T(\pi):\, \pi\in D([0,T],\mc M)\; \hbox{ and }\; \pi_T\notin \mc B_{\varepsilon}(\overline\vartheta) \right\} > 0\, .
\end{eqnarray*}
\end{lemma}
\begin{proof}
Let $\delta=\varepsilon/\sqrt{2}$ and consider $T_0>0$ satisfying \eqref{eq1}. Set $T=T_0$ and assume that the statement of the corollary does not hold. In that case, since $I_T$ has compact level sets, by Lemmas \ref{lemsld} and \ref{rfhe}, there exists a sequence of trajectories $\{\pi^k:\;k\geq 1\}$ in $\mc C([0,T],\mc M^0)$, with $\pi^k_T\notin \mc B_{\varepsilon}(\overline\vartheta)$, converging to some $\pi$ whose density is a weak solution of $\eqref{f02}$. By \eqref{eq1} and since $\bb B_{\delta/2}(\bar\rho)\subset\mc B_{\varepsilon/2}(\overline\vartheta)$, $\pi_T$ belongs to $\mc B_{\varepsilon/2}(\bar\vartheta)$. Hence, for every integer $k>0$,
$$
d\big(\pi^k_T,\pi_T\big)>\varepsilon/2\, ,
$$
which contradicts the convergence of $\pi^k$ to $\pi$ in $C([0,T],\mc M^0)$.
\end{proof}

We conclude this section with an estimate on the cost paying by backwards solutions of the hydrodynamic equation \eqref{f02}. Fix a weak solution $\rho$ of \eqref{f02} and denote by $\pi(t,du) =\tilde\rho(t,u)du$ the path in $C([0,T],\mc M^0)$ with density given by $\tilde\rho(t,du)=\rho(T-t,u)du$. It is clear that $\mc Q_T(\pi) = \mc Q_T(\rho_t(u)du)<\infty$ and that $\partial_t\tilde\rho_t = -\partial_t\rho_{T-t}$. Hence, by \eqref{jflm} and since $\rho$ is a weak solution of \eqref{f02}, for each $G$ in $\mc C^{1,2}_0(\overline{\Omega_T})$,
\begin{eqnarray*}
J_G(\pi) & = & \int_0^T\<\partial_t\tilde\rho_t,G_t\>_{-1,1}\; dt + \int_0^Tdt\int_{\Omega}du\;\nabla\varphi(\tilde\rho_t(u))\cdot\nabla G_t(u) 
\\
& & -\frac{1}{2}\int_0^T dt\int_{\Omega} du\;\sigma(\tilde\rho_t(u))\Vert\nabla G_t(u)\Vert^2
\\
&=&-\int_0^T\<\partial_t\rho_t,\widehat G_t\>_{-1,1}\; dt + \int_0^Tdt\int_{\Omega}du\;\nabla\varphi(\rho_t(u))\cdot\nabla \widehat G_t(u) 
\\
& & -\frac{1}{2}\int_0^T dt\int_{\Omega} du\;\sigma(\rho_t(u))\Vert\nabla \widehat G_t(u)\Vert^2
\\
& = & 2\!\int_0^T \!\!\!dt \!\int_{\Omega}\! du\; \nabla\varphi(\rho_t(u))\cdot\nabla\widehat{G}_t(u) - \frac{1}{2}\!\int_0^T \!\!\!dt \!\int_{\Omega} \!du\; \sigma(\rho_t(u))\Vert\nabla\widehat{G}_t(u)\Vert^2
\\
& \leq & \int_0^T dt\int_{\Omega} du\; \varphi'(\rho_t(u))\frac{\Vert\nabla\rho_t(u)\Vert^2}{\chi(\rho_t(u))}\, ,
\end{eqnarray*}
where $\widehat{G}(t,u) = G(T-t,u)$ and where the last inequality follows from the elementary inequality $2ab\leq Aa^2 + A^{-1}b^2$. In particular, from the definition of $\mc E_T(\rho)$ given in Subsection \ref{subsecenest} and since $\varphi'$ is bounded above on $[0,1]$ by some constant $C_0>0$,
\begin{eqnarray}\label{rfe}
I_T(\pi)\leq C_0\mc E_T(\rho)\, .
\end{eqnarray}

%%%%%%%%%%%%%%%%%%%%%%%%%%%%%%%%%%%%%%%%%%%%%%%%%%%%%%%%%%%%%%%%%%%%%%%%%%%%%%%%%%%%%%%%%%%%%%%%%%%

\subsection{The statical rate function}
\label{secsrf}
%%%%%%%%%%%%%%%%%%%%%%%%%%%%%%%%%%%%%%%%%%%%%%%%%%%%%%%%%%%%%%%%%%%%%%%%%%%%%%%%%%%%%%%%%%%%%%%%%%%

In this subsection we study some properties of the quasi potential $V$. The first main result, presented in Theorem \ref{th2sld}, states that $V$ is continuous at $\bar\rho$ in the $L^2(\Omega)$ topology. The second one, presented in Theorem \ref{lscsld}, states that $V$ is lower semicontinuous.

We start with an estimate on $V(\vartheta)$ which is the main ingredient in the proof of the former.
Let $\bb V:\bb D\to \bb [0,+\infty]$ be the functional given by $\bb V(\rho) = V(\rho(u)du)$. For each $h>0$ and each $\delta>0$, let $\bb D^h_{\delta}$ be the subset of $\bb D$ consisting of those profiles $\rho$ satisfying the following conditions:
\begin{enumerate}
 \item[\bf{\it i)}] $\rho$ belongs to $\in H^1(\Omega)$ and $B\rho = b$.
 \item[\it ii)] $\int_{\Omega}\Vert\nabla\rho(u)\Vert^2du\leq h$.
 \item[\it iii)] $\delta\leq\rho(u)\leq 1-\delta$ a.e. in $\Omega$.
\end{enumerate}

\begin{lemma}\label{quapotest}
For every $h>0$ and every $\delta>0$, there exists a constant $C>0$ such that
\begin{eqnarray*}
\bb V(\rho)\leq C\left\{\Vert\rho-\bar\rho\Vert_2^2\int_0^1\alpha'(t)^2 dt + \int_0^1\alpha(t)^2 dt\right\}
\end{eqnarray*}
for any $\rho$ in $\bb D^h_{\delta}$ and any increasing $\mc C^1$-diffeomorphism $\alpha:[0,1]\to[0,1]$.
\end{lemma}
\begin{proof}
Fix $h>0$ and $\delta>0$. Let $\rho\in\bb D^h_{\delta}$ and let $\alpha:[0,1]\to[0,1]$ be an increasing $\mc C^1$-diffeomorphism. Consider the path $\pi^{\alpha}_t(du)=\rho^{\alpha}(t,u)du$ in $\mc C([0,T],\mc M^0)$ with density given by $\rho^{\alpha}_t = (1-\alpha(t))\bar\rho+\alpha(t)\rho$. From condition  $i)$, it is clear that $\rho^{\alpha}$ belongs to $L^2([0,1],H^1(\Omega))$ (which implies that $\mc Q_1(\pi^{\alpha})<\infty$) and that $B\rho^{\alpha}_t=b$ for every $t$ in $[0,1]$. Further, since $\bar\rho$ solves \eqref{f01}, for every function $G$ in $C^{1,2}_0(\overline{\Omega_1})$ and every $t$ in $[0,1]$,
\begin{eqnarray*}
\int_{\Omega} \nabla\varphi(\rho_t^{\alpha}(u))\cdot\nabla G_t(u) du & = & \int_{\Omega} \nabla[\varphi(\rho_t^{\alpha}(u))-\varphi(\bar\rho(u))]\cdot\nabla G_t(u) du
\\ & = & \int_{\Omega} \Psi_t^{\alpha}(u)\cdot\nabla G_t(u) du \, ,
\end{eqnarray*}
where $\Psi^{\alpha}_t=\alpha(t)\varphi'(\rho_t^{\alpha})\nabla(\rho-\bar\rho) +[\varphi'(\rho_t^{\alpha})-\varphi'(\bar\rho)]\nabla\bar\rho$. From the definition of $\rho^{\alpha}$ it is easy to see that $\partial_t\rho^{\alpha}(t,u) = \alpha'(t)(\rho(u)-\bar\rho(u))$. Hence, by \eqref{jflm},
\begin{eqnarray}\label{eq3}
\begin{aligned}
 J_G(\pi^{\alpha}) \; = \;& \int_0^1 \alpha'(t)\langle\rho-\bar\rho,G_t\rangle\, dt + \int_0^1 dt\,\int_{\Omega} du\, \Psi_t^{\alpha}(u)\cdot\nabla G_t(u)
\\
& \;- \frac{1}{2}\int_0^1 \<\sigma(\rho^{\alpha}_t),\Vert\nabla G_t\Vert^2\> \, dt\, .
\end{aligned}
\end{eqnarray}

Recall that $\bar\rho$ is bounded away from $0$ and $1$. Therefore, from condition $iii)$, there exists a constant $C_1=C_1(\delta)>0$ such that the third term on the right hand side of \eqref{eq3} is bounded above by
\begin{eqnarray*}
-C_1\int_0^1dt\int_{\Omega} du\,\Vert\nabla G_t(u)\Vert^2 \, .
\end{eqnarray*}

On the other hand, by the inequality $2ab\leq Aa^2+A^{-1}b^2$ and by Poincar\'e's inequality, there exists a constant $C_2>0$ such that the first term on the right hand side of \eqref{eq3} is bounded by
\begin{eqnarray*}
C_2\Vert\rho-\bar\rho\Vert_2^2\int_0^1 \alpha'(t)^2 dt + \frac{C_1}{2}\int_0^1dt\int_{\Omega} du\; \Vert\nabla G_t(u)\Vert^2 \, .
\end{eqnarray*}

Finally, from condition $ii)$ and since $\varphi'$ is bounded and Lipschitz on $[0,1]$, there is a constant $C'=C'(h)>0$ such that $\int_{\Omega} \Vert\Psi^{\alpha}_t(u)\Vert^2 du\leq C'\alpha(t)^2$ for every $t$ in $[0,1]$. Hence, by the inequality $2ab\leq Aa^2+A^{-1}b^2$ and by Schwarz inequality, there exists a constant $C_3 = C_3(h,\delta)>0$ such that the second term on the right hand side of \eqref{eq3} is bounded by
\begin{eqnarray*}
C_3\int_0^1\alpha(t)^2dt + \frac{C_1}{2}\int_0^1dt\int_{\Omega} du\;\Vert\nabla G_t(u)\Vert^2\, .
\end{eqnarray*}

Adding these three bounds, we obtain that
\begin{eqnarray*}
 J_G(\pi^{\alpha})\leq C_2\Vert\rho-\bar\rho\Vert_2^2\int_0^1\alpha'(t)^2dt +  C_3\int_0^1\alpha(t)^2 dt\, 
\end{eqnarray*}
for any function $G$ in $\mc C^{1,2}_0(\overline{\Omega_1})$, which implies the desired result.
\end{proof}

\begin{theorem}\label{th2sld}
$\bb V$ is continuous at $\bar\rho$.
\end{theorem}

\begin{proof}
We will prove first that the restriction of $\bb V$ to the sets $\bb D^h_{\delta}$ is continuous at $\bar\rho$. Fix then $h>0$ and $\delta>0$. Let $\{\rho^n:\,n\geq 1\}$ be a sequence in $\bb D^h_{\delta}$ converging to $\bar\rho$. By Lemma \ref{quapotest}, there is a constant $C=C(h,\delta)$ such that

\begin{eqnarray*}
\bb V(\rho^n)\leq C\left\{\Vert\rho^n-\bar\rho\Vert_2^2\int_0^1\alpha'(t)^2 dt + \int_0^1\alpha(t)^2 dt\right\}
\end{eqnarray*}
for any integer $n>0$ and any increasing $\mc C^1$-diffeomorphism $\alpha:[0,1]\to[0,1]$.
Thus, by letting $n\uparrow\infty$ and then taking the infimum over all the increasing $\mc C^1$-diffeomorphisms $\alpha:[0,1]\to[0,1]$, we conclude that
\begin{eqnarray*}
\limsup_{n\to\infty}\bb V(\rho^n)\leq C\inf_{\alpha}\left\{\int_0^1\alpha(t)^2 dt\right\} = 0\, .
\end{eqnarray*}

We deal now with the general case. Let $\{\rho^n:\,n\geq 1\}$ be a sequence in $\bb D$ converging to $\bar\rho$. Fix $\varepsilon> 0$. For each integer $n>0$, let $\lambda^n$ be the weak solution of \eqref{f02} starting at $\rho^n$. By Lemma \ref{lem2}, there exist $T = T(\varepsilon)>0$ and $N_0 = N_0(\varepsilon)>0$ such that, for all integer $n>N_0$,
\begin{eqnarray}\label{est3}
\mc E_T(\lambda^n)\leq \varepsilon\, .
\end{eqnarray}
In particular, there exists $T'\leq T_n\leq 2T' = T$ such that
\begin{eqnarray*}
\int_{\Omega} \Vert\nabla\lambda^n_{T_n}(u)\Vert^2du\leq \varepsilon/T'\, .
\end{eqnarray*}
Moreover, by Lemma \ref{lem05sld}, there exists $\delta=\delta(T')>0$ such that $\delta\leq\lambda_{T_n}^n(u)\leq 1-\delta$ for every integer $n>N_0$ and for every $u$ in $\Omega$. Hence, $\lambda^n_{T_n}$ belongs to $\bb D^{\varepsilon/T'}_{\delta}$. Further, by Lemma \ref{lem1-ann},
$$
\sqrt{2}\Vert\rho^n-\bar\rho\Vert_2 \geq \Vert\rho^n-\bar\rho\Vert_1 \geq \Vert\lambda_{T_n}^n-\bar\rho\Vert_1 \geq \Vert\lambda_{T_n}^n-\bar\rho\Vert_2^2\, ,
$$
which implies that $\lambda_{T_n}^n$ also converges to $\bar\rho$ in $L^2(\Omega)$. Therefore, by the first part of the proof,
\begin{eqnarray*}
\lim_{n\to\infty}\bb V(\lambda_{T_n}^n) = 0\, .
\end{eqnarray*}

For each integer $n>0$, let $\pi^n$ be the path in $C([0,T_n],\mc M)$ given by $\pi^n_t(du) =\lambda^n(T_n-t,u)du$. By \eqref{rfe} and \eqref{est3}, for every integer $n>N_0$,
\begin{eqnarray*}
I_{T_n}(\pi^n)\leq C_0 \mc E_{T_n}(\lambda^n)\leq C_0\varepsilon\, .
\end{eqnarray*}

In particular,
\begin{eqnarray*}
\limsup_{n\to\infty}\bb V(\rho^n) \leq \lim_{n\to\infty}\bb V(\lambda_{T_n}^n) +\limsup_{n\to\infty}I_{T_n}(\pi^n)
\leq C_0\varepsilon\, ,
\end{eqnarray*}
which, by the arbitrariness of $\varepsilon$, implies the desired result.
\end{proof}

Similar arguments permit us to show that the quasi potentials of measures in $\mc M^0$ are uniformly bounded.
\begin{proposition}\label{quapotfin}
$V(\vartheta)$ is finite if and only if $\vartheta$ belongs to $\mc M^0$. Moreover,
\begin{eqnarray*}
\sup_{\vartheta\in\mc M^0}V(\vartheta) < \infty\, .
\end{eqnarray*}
\end{proposition}
\begin{proof}
By Lemmas \ref{lem05sld} and \ref{lem2}, there exist constants $\delta>0$ and $h>0$ such that, for every weak solution $\lambda$ of the equation \eqref{f02},
\begin{eqnarray}\label{quapotfin1}
\delta\leq \lambda(t,u)\leq 1-\delta\;\; \forall (t,u) \in [1,\infty)\times\Omega \;\;\;\hbox{ and }\;\;\; \mc E_2(\lambda)\leq h\, .
\end{eqnarray}

Fix $\vartheta(du)= \rho(u)du$ in $\mc M^0$ and let $\lambda$ be the weak solution of \eqref{f02} starting at $\rho$. By \eqref{quapotfin1}, there exists a time $1\leq T\leq 2$ such that $\lambda_T$ belongs to $\bb D^h_{\delta}$. Moreover, if we denote by $\pi$ the path in $C([0,T], \mc M^0)$ given by $\pi(t,du) = \lambda(T-t,u)du$, by \eqref{rfe}, there exists a constant $C_1 =C_1(h)>0$ such that $I_T(\pi)\leq C_1$. Hence, by Lemma \ref{quapotest}, there exists a constant $C_2=C_2(h,\delta)>0$ such that
\begin{eqnarray*}
\bb V(\rho)\leq \bb V(\lambda_T) + I_T(\pi) \leq C_2+C_1\, .
\end{eqnarray*}
\end{proof}
For any real numbers $r<s$ and any trajectory $\pi$ in $D([\tilde r,\tilde s], \mc M)$ with $\tilde r\leq r\leq s\leq \tilde s$, let $\pi^{[r,s]}$ be the trajectory in $D([0,s-r], \mc M)$ given by $\pi^{[r,s]}_t = \pi_{t+r}$, and let 
$$
I_{[r,s]}(\pi) = I_{s-r}\big(\pi^{[r,s]}\big)\,.
$$

For each $\pi$ in $D((-\infty,0],\mc M)$, let $I_{\pi}:(-\infty,0]\to[0,+\infty]$ be the function given by
$$
I_{\pi}(t) = I_{[t,0]}(\pi)\, .
$$
It is clear that this is a nonincreasing function and then
$$
I(\pi) = \lim_{t\downarrow-\infty}I_{\pi}(t) \in [0,+\infty]\, 
$$
is well defined. We claim that, for every path $\pi$ in $D((-\infty,0],\mc M)$,
\begin{eqnarray}\label{est2sld}
I(\pi)\geq V(\pi_0)\, .
\end{eqnarray}
Moreover, if $I(\pi)<\infty$ then, as $t\downarrow-\infty$, $\pi_t$ converges to $\overline\vartheta$ in $\mc M^0$.

Indeed, the last assertion is an immediate consequence of Lemma \ref{cor1}. To prove \eqref{est2sld}, we may assume of course that $I(\pi)<\infty$. In that case, $\pi(t,du) =\rho(t,u)du$ belongs to $\mc C((-\infty, 0], \mc M^0)$ and, by Lemma \ref{lem1}, there exists a sequence of nonpositive times $\{t_n:\, n\geq 1\}$ such that, for each integer $n>0$, $\rho_{t_n}$ belongs to $\bb B_{1/n}(\bar\rho)$. Hence, for all integer $n>0$,
$$
V(\pi_0) \leq \bb V(\rho_{t_n}) + I_{\pi}(t_n) \leq \bb V(\rho_{t_n}) + I(\pi)\, .
$$
To conclude the proof of \eqref{est2sld} it remains to let $n\uparrow\infty$ and to apply Theorem \ref{th2sld}.

As a consequence of these facts we recover the definition for the quasi potential given in \cite{bdgjl2}, in which the infimum appearing in the definition of $V(\vartheta)$ is carried over all paths $\pi$ in $D([-\infty,0],\mc M)$ with $\pi_{-\infty}=\bar\vartheta$ and $\pi_0 = \vartheta$.

\begin{theorem}\label{lscsld}
The functional $V$ is lower semicontinuous.
\end{theorem}

\begin{proof}
Since $V(\vartheta)$ is finite only for measures $\vartheta$ in $\mc M^0$, which is a closed subset of $\mc M$, we just need to prove that, for all $q \in\bb R_+$, the set
$$
V_q =\{\vartheta\in\mc M^0: V(\vartheta)\leq q\}\, ,
$$
is closed in $\mc M^0$. Fix then $q\in \bb R_+$ and let $\{\vartheta^n(du)=\rho^n(u)du:\,n\geq 1\}$ be a sequence of measures in $V_q$ converging to some $\vartheta(du)=\rho(u)du$ in $\mc M^0$.

By definition of $V$, for each integer $n>0$, there exists a path $\pi^n$ in $C([0,T_n],\mc M^0)$ with $\pi^n_0=\overline\vartheta$, $\pi^n_{T_n} = \vartheta^n$ and such that
\begin{eqnarray}\label{lsc}
I_{T_n}(\pi^n) \leq V(\vartheta^n) + 1/n\leq q +1/n
\end{eqnarray}

Let us assume first that the sequence $\{T_n:\, n\geq 1\}$ is bounded above by some $T>0$. In that case, for each integer $n>0$, let $\hat{\pi}^n$ be the path in $C([0,T],\mc M^0)$ obtained from $\pi^n$ by staying a time $T-T_{n}$ at $\overline\vartheta$,
$$
\hat{\pi}^n_t \;=\;
\begin{cases}

\overline\vartheta & \hbox{if $0\leq t\leq T-T_{n}$}\,,\\

\pi^n_{t-T+T_n} & \hbox{if $T-T_{n}\leq t\leq T$}\,.

\end{cases}
$$

It is clear that $I_T(\hat{\pi}^n|\bar\rho) = I_{T_n}(\pi^n|\bar\rho)$. Moreover, from \eqref{lsc} and since $I_T(\cdot|\bar\rho)$ has compact level sets, there exists a subsequence of $\hat{\pi}^n$ converging to some $\pi$ in $C([0,T],\mc M^0)$ such that $\pi_T(du) =\rho(u)du$ and $I_T(\pi|\bar\rho)\leq q$. In particular, $\rho$ belongs to $V_q$ and we are done.

To complete the proof, let us now assume that $T_n$ has a subsequence which converges to $\infty$. We may suppose, without loss of generality that this subsequence is the sequence $T_n$ itself. For each integer $n>0$, let $\tilde{\pi}^n$ be the path in $C([-T_n,0],\mc M^0)$ given by $\tilde{\pi}^n_t =\pi^n_{t+T_n}$.

Since $I_T$ is lower semicontinuous with compact level sets, for any integer $l>0$ and for any subsequence $\{\tilde{\pi}^{n_r}: \, r\geq 1\}$ of $\tilde{\pi}^n$, there exists a subsequence of $\tilde{\pi}^{n_r}$ converging to some $\check \pi^l$ in $C([-l,0],\mc M^0)$ with $\check \pi^l_0 = \rho$ and $I_{[-k,0]}(\check \pi^l) \leq q$. Then, by a Cantor's diagonal argument, we may obtain a path $\check \pi$ in $C((-\infty,0],\mc M^0)$ with $\check \pi_0 = \rho$ and $I(\check \pi)\leq q$. This together with \eqref{est2sld} conclude the proof of the theorem.
\end{proof}

%%%%%%%%%%%%%%%%%%%%%%%%%%%%%%%%%%%%%%%%%%%%%%%%%%%%%%%%%%%%%%%%%%%%%%%%%%%%%%%%%%%%%%%%%%%%%%%%

\section{Large Deviations}

%%%%%%%%%%%%%%%%%%%%%%%%%%%%%%%%%%%%%%%%%%%%%%%%%%%%%%%%%%%%%%%%%%%%%%%%%%%%%%%%%%%%%%%%%%%%%%%%

We prove here that the quasi potential is the large deviation functional of the stationary measure.
Recall from Section 2 that $\bar\rho$ denotes the weak solution of \eqref{f01} and $\overline\vartheta$, the measure in $\mc M^0$ with density $\bar\rho$, i.e., $\overline\vartheta(du) = \bar\rho(u)du$.

%%%%%%%%%%%%%%%%%%%%%%%%%%%%%%%%%%%%%%%%%%%%%%%%%%%%%%%%%%%%%%%%%%%%%%%%%%%%%%%%%%%%%%%%%%%%%%%%%%%%%%%%%%%
\subsection{Lower Bound}
%%%%%%%%%%%%%%%%%%%%%%%%%%%%%%%%%%%%%%%%%%%%%%%%%%%%%%%%%%%%%%%%%%%%%%%%%%%%%%%%%%%%%%%%%%%%%%%%%%%%%%%%%%%

The proof of the lower bound is essentially the same as the one in \cite{BG}, Subsectin 3.1, but for the sake of completeness we present here the detailed proof. In fact, it is a simple consequence of the hydrostatics \eqref{hs} and the dynamical large deviation lower bound.

Fix an open subset $\mc O$ of $\mc M$. We have to prove that
\begin{eqnarray*}
\liminf_{N\to\infty}\frac{1}{N^d}\log\mc P_N(\mc O) \geq -\inf_{\vartheta\in \mc O}V(\vartheta) \, .
\end{eqnarray*}

To this end, it is enough to show that for any measure $\vartheta$ in $\mc O\cap \mc M^0$ and any trajectory $\tilde\pi$ in $C([0,T],\mc M^0)$ with $\tilde\pi_T = \vartheta$,
\begin{eqnarray}\label{lb1}
\liminf_{N\to\infty}\frac{1}{N^d}\log\mc P_N(\mc O)\geq -I_T(\tilde\pi|\bar\rho)
\end{eqnarray}
holds. Since $\mu^N_{ss}$ is stationary,
$$
\mc P_N(\mc O) = \bb E_{\mu^{N}_{ss}}\left[\bb P_{\eta_0}\left(\pi^N_T\in\mc O\right)\right]\, .
$$

The hydrostatic result \ref{hs} is equivalent to the existence of a sequence of positive numbers $\{\varepsilon_N:\; N\geq 1\}$ converging to $0$ and such that $\mc P_N(\mc B_{\varepsilon_N}(\overline\vartheta))$ converges to 1. Hence, for $N$ large enough,
$$
\mc P_N(\mc O)\geq \frac{1}{2}\inf_{\eta\in\mc B_N}\left\{\bb P_{\eta}\left[\pi^N_T\in \mc O\right]\right\},
$$
where $\mc B_N = \left(\pi^N\right)^{-1}(\mc B_{\varepsilon_N}(\overline\vartheta))$. For each integer $N>0$, consider a configuration $\eta^N$ in $\mc B_N$ satisfying
$$
\bb P_{\eta^N}\left[\pi^N_T\in \mc O\right] = \inf_{\eta\in\mc B_N}\left\{\bb P_{\eta}\left[\pi^N_T\in \mc O\right]\right\}\, .
$$

Since $\pi^N\left(\eta^N\right)$ converges to $\bar\rho$ in $\mc M$ and since $\mc O_T = {\pi_T}^{-1}(\mc O)$ is an open subset of $D([0,T],\mc M)$, by the dynamical large deviations lower bound,
\begin{eqnarray*}
\liminf_{N\to\infty}\frac{1}{N^d}\log\mc P_N(\mc O) & \geq & \liminf_{N\to\infty}\frac{1}{N^d}\log\bb P_{\eta^N}\left[\pi^N_T\in \mc O\right]
\\
 & = & \liminf_{N\to\infty}\frac{1}{N^d}\log {\text{\bf Q}}_{\eta^N}(\mc O_T)
\\
 & \geq & -\inf_{\pi\in \mc O_T} I_T(\pi|\bar\rho) \geq -I_T(\tilde\pi|\bar\rho)\, ,
\end{eqnarray*}
which proves \eqref{lb1} and we are done.

%%%%%%%%%%%%%%%%%%%%%%%%%%%%%%%%%%%%%%%%%%%%%%%%%%%%%%%%%%%%%%%%%%%%%%%%%%%%%%%%%%%%%%%%%%%%%%%%%%%%%%%%%%
\subsection{Upper Bound}
%%%%%%%%%%%%%%%%%%%%%%%%%%%%%%%%%%%%%%%%%%%%%%%%%%%%%%%%%%%%%%%%%%%%%%%%%%%%%%%%%%%%%%%%%%%%%%%%%%%%%%%%%%

In this subsection we prove the upper bound. We follow closely the approach given in \cite{BG} and solve the missing case mentioned in the introduction.

Fix a closed subset $\mc C$ of $\mc M$. We have to show that
\begin{eqnarray}\label{ub}
\limsup_{N\to\infty}\frac{1}{N^d}\log\mc P_N(\mc C) \leq - V(\mc C)\, ,
\end{eqnarray}
where $V(\mc C) = \inf_{\vartheta\in \mc C}V(\vartheta)$.

Notice that if $\overline\vartheta$ belongs to $\mc C$, $V(\mc C) =0$ and the upper bound is trivially verified. Thus, we may assume that $\overline\vartheta\notin\mc C$.

We may also assume that the left hand side of \eqref{ub} is finite, which implies that $\mc C\cap X_N\neq \emptyset$ for infinitely many integers $N$. By the compactness of $\mc M$ and since $\mc C$ is a closed subset of $\mc M$, there exists a sequence of configurations $\{\eta^{N_k}:k\geq 1\}$ with $\pi^N(\eta^{N_k})$ in $\mc C\cap X_{N_k}$ converging to some $\vartheta$ in $\mc C$. Moreover, since each configuration in $X_N$ has at most one particle per site, $\vartheta$ belongs to $\mc M^0$. In particular, by Proposition \ref{quapotfin}, $V(\mc C)<\infty$.

Fix $\delta>0$ such that $\mc B_{3\delta}(\overline\vartheta)\cap\mc C =\emptyset$. Let $B = B_{\delta}$ be the open $\delta$-ball with centre $\overline\vartheta$ in the $d$-metric,
$$
B=\mc B_{\delta}(\overline\vartheta)\, ,
$$
and let $R = R_{\delta}$ be the subset of $\mc M$ defined by
$$
R=\{\vartheta\in\mc M:\, 2\delta \leq d(\vartheta,\overline\vartheta) \leq 3\delta\}\, .
$$

For each integer $N>0$ and each subset $A$ of $\mc M$, let $A^N =(\pi^N)^{-1}(A)$ and let $H_A^N:D(\bb R_+, X_N)\to[0,+\infty]$ be the entry time in $A^N$:
$$
H^N_A = \inf\left\{t\geq 0: \eta_t\in A^N\right\}\,.
$$

Let $\partial B^N = \partial B^N_{\delta}$ be the set of configurations $\eta$ in $X_N$ for which there exists a finite sequence of configurations $\{\eta^i: 0\leq i\leq k\}$ in $X_N$ with $\eta^0$ in $R^N$, $\eta^k = \eta$ and such that

\begin{enumerate}
\item[\it i)] For every $1\leq i\leq k$, the configuration $\eta^{i}$ can be obtained from $\eta^{i-1}$ by a jump of the dynamics.
\item[\it ii)] The unique configuration of the sequence that can enter into $B^N$ after a jump is $\eta^k$.
\end{enumerate}

Let also $\tau_1 = \tau_1^N: D(\bb R_+, X_N)\to [0,\infty]$ be the stopping time given by

$$
\tau_1 = \inf\left\{t> 0: \hbox{ there exists } s < t \hbox{ such that } \eta_s\in R^N \hbox{ and } \eta_t\in\partial B^N\right\}\, .
$$

The sequence of stopping times obtained by iterating $\tau_1$ is denoted by $\tau_k$. This sequence generates a Markov chain $X_k$ on $\partial B^N$ by setting $X_k = \eta_{\tau_k}$.

Notice that this Markov chain is irreducible. In fact, let $\zeta,\eta$ be configurations in $\partial B^N$. By definition of the set $\partial B^N$, there exist a sequence $\{\eta^i: 0\leq i\leq k\}$ in $X_N$ which satisfies $\eta^0\in R^N$, $\eta^k = \eta$, $i)$ and $ii)$. Further, it is clear that there exists a sequence $\{\zeta^i: 0\leq i\leq l\}$ in $X_N$ which satisfies $\zeta^0 = \zeta$, $\zeta^l = \eta^0$ and $i)$. Consider then the sequence $\{\tilde\eta^j:0\leq j\leq l+k\}$ in $X_N$ given by
$$
\tilde\eta^j \;=\;
\begin{cases}

\zeta^j & \hbox{ if }\; 0\leq j\leq l\, ,\\

\eta^{j-l} & \hbox{ if }\; l < j\leq l+k\, .

\end{cases}
$$

Let $j_0 = 0$ and for $i\geq 1$ let
$$
j_{2i-1} = \min_{j> j_{2i-2}\atop \eta^j\in R^N}\{j\}\qquad\hbox{ and }\qquad
j_{2i} = \min_{j>j_{2i-1}\atop \eta^j\in \partial B^N}\{j\}\, .
$$

Thus, by setting $\xi^i = \tilde\eta^{j_{2i}}$, we obtain a sequence $\{\xi^i:0\leq i\leq r\}$ in $\partial B^N$ starting at $\xi^0 = \zeta$, ending at $\xi^r = \eta$ and such that
$$
\bb P_{\xi_{i-1}}[\eta_{\tau_1} = \xi_{i}] > 0\, ,
$$
for every $1\leq i\leq r$. This implies the irreducibility of $X_k$.

Hence, since the state space $\partial B^N$ is finite, this Markov chain has a unique stationary measure $\nu_N$. Following \cite{BG, FW}, we represent the stationary measure $\mu^N_{ss}$ of a subset $A$ of $X_N$ as
$$
\mu^N_{ss}(A) = \frac{1}{C_N}\int_{\partial B^N}\bb E_{\eta}\left(\int_0^{\tau_1} {\bf 1}_{\left\{\eta_s\in A\right\}}ds\right)d\nu_N(\eta)\, ,
$$
where
$$
C_N=\int_{\partial B^N}\bb E_{\eta}\left(\tau_1\right)d\nu_N(\eta)\, .
$$

In particular, by this representation an by the strong Markov property,
\begin{eqnarray*}
\mc P_N(\mc C)  \leq  \frac{1}{C_N}\sup_{\eta\in\partial B^N}\left\{\bb P_{\eta}\left[H_{\mc C}^N < \tau_1\right]\right\}  \sup_{\eta\in \mc C^N} \left\{\bb E_{\eta}\left(\tau_1\right)\right\}\, .
\end{eqnarray*}

Recall that a configuration in $X_N$ can jump by the dynamics to less than other $2dN^d$ configurations and that the jump rates are of order $N^2$. Hence, since any trajectory in $D(\bb R_+,X_N)$ has to perform at least a jump before the stopping time $\tau_1$, $C_N\geq 1/CN^{d+2}$ for some constant $C>0$.

Notice that the jumps of the process $d(\pi^N(\eta_t),\bar\rho)$ are of order $N^{-d}$. Thus, for $N$ large enough, any trajectory in $D(\bb R_+, X_N)$ starting at some configuration in $\partial B^N$, resp. $\mc C^N$, satisfies $H^N_R \leq H^N_{\mc C}$, resp. $\tau_1\leq H^N_B$. Hence, by the strong Markov property,

\begin{eqnarray*}
\mc P_N(\mc C) \leq  CN^{d+2} \sup_{\eta\in R^N} \left\{\bb P_{\eta}\left[H_{\mc C}^N < H_B^N\right]\right\}  \sup_{\eta\in \mc C^N} \left\{\bb E_{\eta}\left(H_B^N\right)\right\}\, .
\end{eqnarray*}
Therefore, in order to prove \eqref{ub}, it is enough to show the next lemma.

\begin{lemma}\label{lem3}
For every $\delta >0$,
\begin{eqnarray}\label{est6}
\limsup_{N\to\infty}\frac{1}{N^d}\log\sup_{\eta\in X_N}\left\{\bb E_{\eta}\left(H^N_{B_\delta}\right)\right\} \leq 0\, .
\end{eqnarray}

For every $\varepsilon>0$, there exists $\delta>0$ such that
\begin{eqnarray}\label{est7sld}
\limsup_{N\to\infty}\frac{1}{N^d}\log\sup_{\eta\in R^N_{\delta}}\left\{\bb P_{\eta}\left[H_{\mc C}^N < H_{B_{\delta}}^N\right]\right\} \leq -V(\mc C) +\varepsilon\, .
\end{eqnarray}
\end{lemma}

To prove this lemma, we will need the following technical result.

\begin{lemma}\label{lem4}
For every $\delta>0$, there exists $T_0,C_0,N_0 >0$ such that
$$
\sup_{\eta\in X_N} \left\{\bb P_{\eta}\left[H_{B_{\delta}}^N \geq kT_0\right]\right\} \leq \exp\left\{-kC_0N^d\right\}\, ,
$$
for any integers $N>N_0$ and $k>0$.
\end{lemma}

\begin{proof}
Fix $\delta>0$. By Lemma \ref{cor1}, there exists $T_0>0$ and $C_0>0$ such that
$$
\inf_{\pi\in\mc D} I_{T_0}(\pi) > C_0\, ,
$$
where $\mc D = D([0,T_0],\mc M\backslash B)$. For each integer $N>0$, consider a configuration $\eta^N$ in $X_N$ such that
$$
\bb P_{\eta^N}\left[H_B^N \geq T_0\right] = \sup_{\eta\in X_N} \left\{\bb P_{\eta}\left[H_B^N \geq T_0\right]\right\}\, .
$$

By the compactness of $\mc M$, every subsequence of $\pi^N(\eta^N)$ contains a subsequence converging to some $\vartheta$ in $\mc M$. Moreover, since each configuration in $X_N$ has at most one particle per site, $\vartheta$ belongs to $\mc M^0$. From this and since $\mc D$ is a closed subset of $D([0,T_0],\mc M)$, by the dynamical large deviations lower bound, there exists a measure $\vartheta(du) = \gamma(u)du$ in $\mc M^0$ such that
\begin{eqnarray*}
\limsup_{N\to\infty}\frac{1}{N^d}\log \bb P_{\eta^N}\left[H_B^N\geq  T_0\right]  & = &
\limsup_{N\to\infty}\frac{1}{N^d}\log {\text{\bf Q}}_{\eta^N}(\mc D)
\\
& \leq & -\inf_{\pi\in\mc D} I_{T_0}(\pi|\gamma)
\\
 & < & -C_0\, .
\end{eqnarray*}
In particular, there exists $N_0>0$ such that for every integer $N>N_0$,
\begin{eqnarray*}
\bb P_{\eta^N}\left[H^N_B\geq T_0\right] \leq \exp\{-C_0N^d\}\,.
\end{eqnarray*}
To complete the proof, we proceed by induction. Suppose then that the statement of the lemma is true until an integer $k-1>0$. Let $N>N_0$ and let $\hat{\eta}$ be a configuration in $X_N$. By the strong Markov property,
\begin{eqnarray*}
\bb P_{\hat{\eta}}\left[H^N_B\geq kT_0\right] & = & \bb E_{\hat{\eta}}\left[{\bf 1}_{\left\{H^N_B\geq T_0\right\}}\bb P_{\eta_{ _{T_0}}}\left[H^N_B\geq(k-1)T_0\right]\right]
\\
 & \leq & \bb P_{\hat{\eta}}\left[H_B^N\geq T_0\right]\sup_{\eta\in X_N}\left\{\bb P_{\eta}\left[H^N_B\geq(k-1)T_0\right]\right\}
\\
& \leq & \exp\left\{-kC_0N^d\right\}\, ,
\end{eqnarray*}
which concludes the proof.

\end{proof}

\begin{proof}[\bf Proof of Lemma \ref{lem3}]
Let $\delta>0$ and consider $T_0,C_0,N_0>0$ satisfying the statement of Lemma \ref{lem4}. For every integer $N>N_0$ and every configuration $\eta$ in $X_N$,

\begin{eqnarray*}
\bb E_{\eta}\left(H_B^N\right)\leq T_0\sum_{k=0}^{\infty}\bb P_{\eta}\left(H_B^N\geq kT_0\right) \leq T_0\sum_{k=0}^{\infty}\exp\left\{-kC_0N^d\right\} \leq \frac{T_0}{1-e^{-C_0}}\, ,
\end{eqnarray*}
which proves \eqref{est6}.

We turn now to the proof of \eqref{est7sld}. Fix $\varepsilon>0$. By Lemma \ref{lem4} and since $V(\mc C)<\infty$, for every $\delta>0$, there exists $T_{\delta}>0$ such that
\begin{eqnarray*}
\limsup_{N\to\infty}\frac{1}{N^d}\log\sup_{\eta\in X_N}\left\{\bb P_{\eta}\left[T_{\delta} \leq H_{B_{\delta}}^N\right]\right\} \leq -V(\mc C)\,.
\end{eqnarray*}

For each integer $N>0$, consider a configuration $\eta^N$ in $R^N_{\delta}$ such that
$$
\bb P_{\eta^N}\left[H^N_{\mc C}\leq T_{\delta}\right] = \sup_{\eta\in R^N_{\delta}}\left\{\bb P_{\eta}\left[H^N_{\mc C}\leq T_{\delta}\right]\right\}\, .
$$

Let $\mc C_{\delta}$ be the subset of $D([0, T_{\delta}],\mc M)$ consisting of all those paths $\pi$ for which there exists $t$ in $[0,T_{\delta}]$ such that $\pi(t)$ or $\pi(t-)$ belongs to $\mc C$. Notice that $\mc C_{\delta}$ is the closure of $\pi^N(\{H^N_{\mc C}\leq T_{\delta}\})$ in $D([0,T_{\delta}],\mc M)$.

Recall that every subsequence of $\pi^N(\eta^N)$ contains a subsequence converging in $\mc M$ to some $\vartheta$ that belongs to $\mc M^0$. Hence, by the dynamical large deviations upper bound, there exists a measure $\vartheta_{\delta}(du)=\gamma_{\delta}(u)du$ in $R_{\delta}\cap\mc M^0$ such that
\begin{eqnarray*}
\limsup_{N\to\infty}\frac{1}{N^d}\log\bb P_{\eta^N}\big(H^N_{\mc C}\leq T_{\delta}\big)\leq \limsup_{N\to\infty}\frac{1}{N^d}\log{\text{\bf Q}}_{\eta^N}(\mc C_{\delta})\leq -\inf_{\pi\in\mc C_{\delta}}I_{T_{\delta}}(\pi|\gamma_{\delta})\, .
\end{eqnarray*}
Therefore, since
$$
\limsup_{N\to\infty}\frac{1}{N^d}\log\{a_N+b_N\}\leq \max\left\{\limsup_{N\to\infty}\frac{1}{N^d}\log a_N, \limsup_{N\to\infty}\frac{1}{N^d}\log b_N\right\}\, ,
$$
the left hand side in \eqref{est7sld} is bounded above by
$$
\max\left\{-V(\mc C),-\inf_{\pi\in\mc C_{\delta}}I_{T_{\delta}}(\pi|\gamma_{\delta})\right\}
$$
for every $\delta>0$. Thus, in order to conclude the proof, it is enough to check that there exists $\delta>0$ such that
$$
\inf_{\pi\in\mc C_{\delta}}I_{T_{\delta}}(\pi|\gamma_{\delta})\geq V(\mc C) -\varepsilon\,.
$$

Assume that this is not true. In that case, for every integer $n>0$ large enough, there exists a path $\pi^n$ in $\mc C_{1/n}\cap C([0,T_{1/n}],\mc M^0)$ such that
$$
I_{T_{1/n}}(\pi^n|\gamma_{1/n})< V(\mc C)-\varepsilon\,.
$$
Moreover, since $\pi^n$ belongs to $\mc C_{1/n}\cap C([0,T_{1/n}],\mc M^0)$, there exists $0<\widetilde T_n\leq T_{1/n}$ such that $\pi^n_{\widetilde T_n}$ belongs to $\mc C$.

Let us assume first that the sequence of times $\{\widetilde T_n: \, n\geq 1\}$ is bounded above by some $T>0$. For each integer $n>0$, let $\tilde\pi^n$ be the path in $C([0,T],\mc M^0)$ given by
$$
\tilde\pi^n_t \;=\;
\begin{cases}

\pi^n_t & \hbox{ if }\; 0\leq t\leq \widetilde T_n\,,\\

\pi^n_{\widetilde T_n} & \hbox{ if }\; \widetilde T_n\leq t\leq T\, .

\end{cases}
$$
Since $I_T$ has compact level sets and since $\pi^n_0(du) = \gamma_{1/n}(u)du$ belongs to $R_{1/n}\cap \mc M^0$ for every integer $n> 0$, we may obtain a subsequence of $\tilde\pi^n$ converging to some $\pi$ in $C([0,T],\mc M^0)$ such that $\pi_0 = \overline\vartheta$, $\pi_T \in\mc C$ and $I_T(\pi)\leq V(\mc C)-\varepsilon$, which contradicts the definition of $V(\mc C)$ and we are done.

To complete the proof, let us assume now that there exists a subsequence $\big\{\widetilde T_{n_k}:\, k\geq 1\big\}$ of $\widetilde T_n$ converging to $\infty$. By Theorem \ref{th2sld}, there exists $\delta>0$ such that $\bb V(\rho)<\varepsilon$ for every $\rho$ in $\bb B_{\delta}(\bar\rho)$. Moreover, if $\pi^{n_k}_{t}(du) = \rho^{n_k}(t,u)du$, by Lemma \ref{lem1}, for any integer $k$ large enough, there exists $0\leq t_k\leq \tilde T_{n_k}$ such that $\rho^{n_k}_{t_k}$ belongs to $\bb B_{\delta}(\bar\rho)$. Then,
\begin{eqnarray*}
V(\pi^{n_k}(\widetilde T_{n_k})) & \leq & \bb V(\rho^{n_k}_{t_k})+ I_{[t_k,\widetilde T_{n_k}]}(\pi^{n_k})
\\
 & < & \varepsilon +V(\mc C) - \varepsilon = V(\mc C)\, ,
\end{eqnarray*}
which also contradicts the definition of $V(\mc C)$ and we are done.
\end{proof}

\section*{Acknowledgements}
I would like to thank my PhD advisor, Claudio Landim, for suggesting this problem, for valuable discussions and support. I also thank Thierry Bodineau for stimulating discussions on this topic.

\end{document}